\documentclass{amsart}

\usepackage{amsfonts}
\usepackage{amssymb, amsmath, mathrsfs}
\usepackage{hhline}
\usepackage{xcolor}

 \newtheorem{thm}{Theorem}[section]
 \newtheorem{cor}[thm]{Corollary}
 \newtheorem{lem}[thm]{Lemma}
 \newtheorem{prop}[thm]{Proposition}
 \theoremstyle{definition}
 \newtheorem{defn}[thm]{Definition}
 \theoremstyle{remark}
 \newtheorem{rem}[thm]{Remark}
 
 \numberwithin{equation}{section}

\usepackage{graphicx,curvesls}
\usepackage{fancyhdr}
\usepackage{color}
\usepackage{colortbl}
\usepackage{pstricks,pst-node,pst-text,pst-3d}

\usepackage{amsfonts}
\usepackage{amssymb, amsmath, mathrsfs}

\usepackage{multicol}

\newrgbcolor{LemonChiffon}{1. 0.98 0.8}
\newrgbcolor{LightBlue}{0.2 1.0 1.0} % (proposto por Jos{\'e} Luis Cardoso)
\newrgbcolor{pink}{1.00 .75 .79}
\newrgbcolor{orange}{1.00 0.65 0.00}
\newrgbcolor{purple}{0.63 0.13 0.94}
\newrgbcolor{myred}{.90 .10 .10}
\newrgbcolor{goldenrod}{.80392 .60784 .11373}
\newrgbcolor{darkgoldenrod}{.5451 .39608 .03137}
\newrgbcolor{brown}{.15 .15 .15}
\newrgbcolor{darkolivegreen}{.33333 .41961 .18431}
\newrgbcolor{siennadois}{.93 .47 .26}

\begin{document}

\title[Gr\"uss inequalities]
{Gr\"uss inequalities for the $\beta-$integral associated \\
with the general quantum operator
%\footnote{The final version of this preprint will
%be published in Journal of Difference Equations and Applications}
}

\author[Cardoso \& Haque \& Macedo]{J.L. Cardoso, N. Haque \& {\^A}. Macedo}

\address{Centro de Matem{\'a}tica, Universidade
do Minho - Polo CMAT-UTAD \\
Dep. de Matem\'atica da Escola de Ci\^encias e Tecnologia \\
Universidade de Tr\'as-os-Montes e Alto Douro (UTAD) \\
Polo
Quinta de Prados \\
5001-801 Vila Real \\
Portugal \\
ORCID ID: 0000-0002-7418-3634
\and
Lalbagh Singhi High Shool \\
Lalbagh, Murshidabad \\
West Bengal - 742149 India \\
\and Centro de Matem{\'a}tica, Universidade do Minho - Polo CMAT-UTAD \\
Dep. de Matem\'atica da Escola de Ci\^encias e Tecnologia \\
Universidade de Tr\'as-os-Montes e Alto Douro (UTAD) \\
Quinta de Prados \\
5001-801 Vila Real \\
Portugal \\
ORCID ID: 0000-0002-5867-3545}
\email{jluis@utad.pt \& nazrul@rkmvccrahara.org \& amacedo@utad.pt}

\thanks{This research was partially supported by Portuguese Funds through FCT
(Funda\c c\~ao para a Ci\^encia e a Tecnologia) within the Projects
UIDB/00013/2020 and UIDP/00013/2020.}



\date{September 14, 2024}

\begin{abstract}
Assume that $\,I\subseteq\mathbb{R}\,$ is an interval and
$\,\beta:\,I\rightarrow\,I\,$ a strictly increasing and continuous
function with a single fixed point $\,s_0\in I\,$, satisfying
$\,(s_0-t)(\beta(t)-t)\leq 0\,$ for all $\,t\in I$, where the
equality occurs only when $\,t=s_0$.

\noindent Hamza et al. considered the general quantum operator,
$\,D_{\beta}[f](t):=\displaystyle\frac{f\big(\beta(t)\big)-f(t)}{\beta(t)-t}\,$
when $\,t\neq s_0\,$ and $\,D_{\beta}[f](s_0):=f^{\prime}(s_0)\,$ when
$\,t=s_0\,$. It generalizes the Jackson $\,q$-derivative operator $\,D_{q}\,$ as
well as the Hahn (quantum derivative) operator, $\,D_{q,\omega}$.

\noindent We obtained Gr\"uss type inequalities for its inverse operator, the
$\beta$-integral. Furthermore, we introduced the concept of
$\,\beta$-Riemann-Stieltjes integral and obtained Gr\"uss type
inequalities associated with it.
\end{abstract}

\maketitle

\section{Introduction}
In the context of the classical Riemann or the Lebesgue integrals, in 1935
\cite{G:1935}, G. Gr\"uss proved that
\begin{equation}\label{Gruss-inequality-product}
\left|\frac{1}{b-a}\int_{a}^{b}f(x)\,g(x)\,dx-
\frac{1}{b-a}\int_{a}^{b}f(x)\,dx\,\frac{1}{b-a}\int_{a}^{b}g(x)\,dx\right|
\leq \frac{1}{4}(M-m)(N-n)
\end{equation}
whenever $\,f\,$ and $\,g\,$ are integrable functions in $\,[a,b]\,$ which satisfy
\begin{equation*}
m\leq f(x)\leq M\quad {\mbox and}\quad n\leq g(x)\leq N\,,
\quad\mbox{for all}\;\:x\in [a,b]\,.
\end{equation*}
New proofs of this result can be found, for instance, in \cite[p. 296]{MPF:1993},
\cite[p. 162]{MPU:2000} or \cite[p. 436]{LMR:2002}. Other publications that offer
improvements or generalizations of the Gr\"uss inequality, or present new results
within the framework of the classical Riemann-Stieltjes integral, can be found in
the literature \cite{D:1998, F:1999, D-Stieltjes:2003, PT:2005, CD:2007, EMP:2007,
AGR:2011, AD-Lipschitz:2014, AD-Monotonic:2014}.
Several applications of this inequality are also documented. For instance,
in \cite{DW:1997}, an estimation of error bounds for special
means and some numerical quadrature rules are provided. In \cite{DF:1998}, some
applications to special means are explored, while
%in \cite{FD:1999}, one finds an application to the Simpson's rule and special means;
in \cite{D:1999}, applications for linear positive functionals and
integrals are discussed. A variety of bounds for the expectation and variance of a
random variable can be found in \cite{CD:2001, BDA:2002, L:2007}.
In \cite{CS:2002}, the authors used a variant of Gr\"uss inequality to obtain
new perturbed trapezoid inequalities. A connection between the Gr\"uss inequality
and the error of best approximation is established in \cite{LMR:2002}.
In \cite{DG:2002}, several examples with inequalities are shown, some
involving probability density functions, and others involving infinite numerical sums.
In \cite{D:2003}, applications for deriving new integral inequalities were achieved,
along with some applications to quadrature rules.
Additionally, in \cite{G:2004}, a $\,q$-Gr\"uss inequality is exhibited in the
restricted $\,q$-integral setting.
In \cite{AGR:2011}, after obtaining a Gr\"uss inequality for a positive linear
operator, the authors applied it to the Berstein operator, to Hermite-Fej{\'e}r
interpolation operator and to convolution-type operators.
In \cite{AD-Monotonic:2014}, Gr\"uss type inequalities were applied in the context
of the classical Riemann-Stieltjes integral for functions of self-adjoint operators
on complex Hilbert spaces.
In \cite{TN:2014}, a $\,q$-Gr\"uss inequality was obtained for Lipschitz functions.
In \cite{LY:2016}, Gr\"uss-type quantum integral inequalities were found
and applied to obtain some related quantum integral inequalities.
In \cite{FMYAEA:2024}, Gr\"uss inequalities were used to generalize some well-known
integral inequalities, such as the Ostrowski, Ostrowski-Gr\"uss, and
Hermite-Hadamard inequalities. Several authors have also published results involving
Ostrowski-Gr\"uss type inequalities, for instance, see \cite{MPU:2000, C:2001}.

The purpose of this work is to obtain new results involving Gr\"uss-type inequalities
in the context of the quantum calculus. To achieve this, we will follow a unified
approach by using the so-called $\,\beta$-difference operator, which will be
introduced later in Subsection \ref{Subsect-2.1}.

The $\beta$-operator $\,D_{\beta}\,$, considered in the abstract, was introduced in
\cite{HSSA:2015}, along with the corresponding general quantum difference calculus.
It is a generalization of the $(q,\omega)$-derivative operator (Hahn's quantum
operator), given by
\begin{equation}\label{Hahn-operator}
D_{q,\omega}[f](x):=\frac{f\big(qx+\omega\big)-f(x)}{(q-1)x+\omega}\,,
\end{equation}
which itself is a generalization of both the Jackson $\,q$-derivative
\begin{equation}\label{Jackson-operator}
D_{q}[f](x):=\frac{f(qx)-f(x)}{(q-1)x}\,,
\end{equation}
and the (forward difference) $\,\omega$-derivative
$$
\triangle_{\omega}[f](x):=\frac{f(x+\omega)-f(x)}{\omega}\,,
$$
where $\,0<q<1$ and $\,\omega>0\,$ are fixed parameters.

Of particular importance are the corresponding inverse operators, namely, respectively:
the Jackson-Thomae-N\"orlund
$\,(q,\omega)$-integral
\begin{equation*}%\label{Jackson-Thomae-Norlund-integral}
\int_a^b f\,{\rm d}_{q,\omega}:=
\int_{\omega_0}^b f\,{\rm d}_{q,\omega}-\int_{\omega_0}^a f\,{\rm d}_{q,\omega}\,,
\end{equation*}
the Jackson $\,q$-integral
\begin{equation*}%\label{Jackson-integral}
\int_{a}^{b}f\,{\rm d}_q:=(1-q)\sum_{k=0}^{+\infty}\big[bf(bq^k)-af(aq^k)\big]q^k
\end{equation*}
and the N\"orlund integral
%\begin{equation}\label{Norlund-sum}
%\int_{\infty}^{x}f\Delta_{\omega}:=-\omega\sum_{k=0}^{+\infty}f(x+k\omega)\,,
%end{equation}
$$
\int_{a}^{b}\!f\,\triangle_{\omega}:=\omega\sum_{k=0}^{+\infty}\big[f(b+k\omega)-f(a+k\omega)\big]\,.
$$
To find out more about this subject see, for example, \cite{AHA:2012}.

These difference operators, as well as their inverse operators, are relevant in pure
mathematics likewise in applied mathematics, being an appeal for multiple investigators
with research production in a wide variety of related subject-matter, such as,
but not confined to,
the quantum calculus, %\cite{AM:2012},
the quantum variational calculus, %\cite{MT:2014},
$\,q$-difference equations properties, %\cite{AAIM:2007},
Sturm-Liouville problems, Paley-Wiener results, Sampling theory,
$\,q-$exponential, trigonometric, hyperbolic and other significant
families of functions associated with Fourier developments and
related properties, combined with their orthogonality
characteristic, functional spaces bonded with inner-products
linked with basic integrals. For more on these topics, see, for
instance, \cite{C:2011, CP:2015, C:2018, AHM:2018, AH:2018,
H:2019, AA-NC:2019} and references therein.
The above operators find their true origin in problems related with the theory of
orthogonal polynomials. Recently, in \cite{ACMP:2020, ACMP:2022}, and using the
arguments of \cite{CM:2020}, it was proved, in a unified way, that the classical
orthogonal polynomials with respect to the Hahn's and the Jackson's operators are
special or limiting cases of a four parametric family of $\,q$-polynomials.

Several publications have addressed the general quantum operator $\,D_{\beta}\,$ or
its inverse.
In 2015, $\,\beta$-H\"older, $\,\beta$-Minkowski, $\,\beta$-Gronwall,
$\,\beta$-Bernoulli and $\,\beta$-Lyapunov inequalities were exhibited \cite{HS:2015}.
In 2016, the existence and uniqueness of solutions of general
quantum difference equations were proven \cite{HS:2016}.
In \cite{HSS:2017}, the exponential, trigonometric and hyperbolic functions were
introduced, while in \cite{FSZ:2018}, the $n$th-order linear general quantum difference
equations were studied. The general quantum variational calculus was developed in \cite{CM:2018},
and a general quantum Laplace transform was designed in \cite{SFZ:2020}.
In \cite{C:2021}, properties of the $\beta$-Lebesgue spaces were generated, while
\cite{KSC:2022} presented the directional derivative within the framework of quantum
calculus.
Hardy-type inequalities were established in \cite{HAA:2023},
Hermite-Hamadard type inequalities were derived in \cite{CS:2024}, and the Taylor
theory within the general quantum setting was developed in \cite{GT:2023, SZ:2024}.

In section \ref{section-2}, we recall the definition of the $\beta$-difference
operator and its inverse, the $\beta$-integral, together with some of
their properties. Section \ref{section-3} is devoted to the $\beta$-Gr\"uss-type
inequality and various approaches to proving it, while in Section \ref{section-4},
we introduce the concept of $\,\beta$-Riemann-Stieltjes integral and obtain multiple
Gr\"uss-type inequalities in this context.

The outcome of this work can be found in sections \ref{section-3}, \ref{section-4}
and \ref{appls}, mainly in the first two of them.
We believe that Theorem \ref{Gruss-inequality-theorem}, Proposition
\ref{positivity-proposition}, Lemma \ref{beta-Korkine},
Propositions \ref{double-Holder-inequality} and \ref{kind-of-Cauchy-Schwarz-proposition},
Corollary \ref{lemma-functional}, Theorem \ref{Gruss-type-inequality-theorem},
Proposition \ref{Abs-integral}, Theorem \ref{Gruss-type-Riemann-Stieltjes-theorem},
Corollaries \ref{corollary-1}, \ref{corollary-2}, \ref{corollary-3}, \ref{corollary-4},
\ref{corollary-5}, \ref{corollary-6}, \ref{Gruss-type-ineq-Jackson} and
\ref{Gruss-type-ineq-Jackson-Thomae}, are original contributions, as well as the
application results obtained in section \ref{appls}.
In our opinion, the most important result is Theorem
\ref{Gruss-type-Riemann-Stieltjes-theorem}.

\section{The $\beta$-difference operator and the $\beta$-integral}\label{section-2}

\subsection{The $\beta$-difference operator}\label{Subsect-2.1}
Henceforth, $\,I\subseteq\mathbb{R}\,$ represents an interval and
$\,\beta:\,I\rightarrow\,I\,$ is a strictly increasing and continuous function,
with a single fixed point $\,s_0\in I\,$, satisfying
\begin{equation}\label{condition1}
(t-s_0)(\beta(t)-t)\leq 0
\end{equation}
for all $\,t\in I$, where the equality is verified only if $\,t=s_0\,$.

Being $\,f:\,I\rightarrow\,\mathbb{K}\,$ a function where
$\,\mathbb{K}\,$ is either $\,\mathbb{R}\,$ or $\,\mathbb{C}\,$
\footnote{In fact, $\mathbb{K}=\mathbb{X}$ can represent any Banach space
\cite[p.2]{HSSA:2015}},
Hamza et al. \cite{HSSA:2015} established the general quantum difference operator
\begin{equation}\label{beta-derivative}
D_{\beta}[f](t):=\left\{
\begin{array}{ccl}
\displaystyle\frac{f\big(\beta(t)\big)-f(t)}{\beta(t)-t} & \mbox{\rm if} &
t\neq s_0\,, \\ [1.2em] f^{\prime}(s_0) & \mbox{\rm if} &
t=s_0\; .
\end{array}\right.,
\end{equation}
presuming that $\,f^{\prime}(s_0)\,$ exists. $\,D_{\beta}[f](t)\,$ is
called the $\beta$-derivative of $\,f\,$ at $\,t\in I$.
If $\,f^{\prime}(s_0)\,$ exists then $\,f\,$ is said to be $\beta$-differentiable on $\,I\,$.
Sometimes we will represent it by $\,\left(D_{\beta}f\right)(t)\,$.

In the case where $\,\beta(t)=qt+\omega$, one gets the Hahn operator
(\ref{Hahn-operator}), being $\,s_0=\frac{\omega}{1-q}\,$ its fixed point.

We mention that it is feasible to replace the above restriction (\ref{condition1})
by $\,(t-s_0)(\beta(t)-t)\geq 0\,$ for $\,t\in I$.

An initiation to the $\,\beta$-calculus associated with this general quantum
difference operator can be found in \cite{HSSA:2015}.

\subsection{The $\beta$-integral}\label{beta-integral}

Taking
$\,\beta^k(t):=(\underbrace{\beta\circ\beta\circ\ldots\circ\beta}_{k\,\mbox{times}})(t)\,$
for $\,t\in I\,$ and $\,k\in \mathbb{N}_0=\mathbb{N}\cup\{0\}$, with $\,\beta^0(t):=t,$
we can build-up the $\,\beta$-interval with extreme points $\,a\,$ and $\,b,$
\begin{equation*}%\label{beta-interval}
[a,b]_{\beta}:=\big\{\,\beta^n(x)\; |\; (x,n)\in
\{a,b\}\times\mathbb{N}_0\,\big\} \cup\{s_0\}.
\end{equation*}
Clearly
$$
a,b\in I\quad\Rightarrow\quad [a,b]_{\beta}\subset I\,,
$$
whensoever $\,a\,$ and $\,b$ are real numbers.

\vspace{0.5em}
The ensuing proposition is stated in \cite[Lemma 2.1, page 3]{HSSA:2015}.

\noindent
\textbf{Proposition A}.
The sequence of functions $\,\left\{\beta^k(t)\right\}_{k\in\mathbb{N}_0}\,$ converges
uniformly to the constant function $\,\hat{\beta}(t):=s_0\,$ on every compact interval
$\,J\subset\,I\,$ containing $\,s_0$.

\vspace{0.5em}
The quantum difference inverse operator, the $\beta$-integral, with $\,a,b\in I,$
is defined by
\begin{equation}\label{beta-integral-a-b}
\int_{a}^{b}f\,{\rm d}_{\beta}:=\int_{s_0}^{b}f\,{\rm d}_{\beta}-\int_{s_0}^{a}f\,{\rm d}_{\beta}
\end{equation}
where
\begin{equation}\label{beta-integral-s0-x}
\int_{s_0}^{x}f\,{\rm d}_{\beta}:=
\sum_{k=0}^{+\infty}\Big(\beta^k(x)-\beta^{k+1}(x)\Big)f\big(\beta^{k}(x)\big)\,.
\end{equation}
Consequently,
\begin{equation}\label{beta-integral-a-b-explicitly}
\int_{a}^{b}f\,{\rm d}_{\beta}=\sum_{k=0}^{+\infty}\Big(\beta^k(b)-\beta^{k+1}(b)\Big)f\big(\beta^{k}(b)\big)-
\sum_{k=0}^{+\infty}\Big(\beta^k(a)-\beta^{k+1}(a)\Big)f\big(\beta^{k}(a)\big)\,.
\end{equation}
Whenever the infinite sum on the right-hand side of (\ref{beta-integral-s0-x})
converges, we say that the function $\,f\,$ is $\,\beta$-integrable on $\,[s_0,x].$
The $\,\beta$-integral on the left-hand side of (\ref{beta-integral-a-b}) is well-defined
on condition that at least one of the $\,\beta$-integrals in the right-hand side is finite.
Moreover, we say that $\,f\,$ is $\,\beta$-integrable on $\,[a,b]\,$ if it is
$\,\beta$-integrable on both $\,[s_0,a]\,$ and $\,[s_0,b].$

\noindent The foremost specific cases are the Jackson-Thomae-N\"orlund integral,
where $\,\beta(t)=qt+\omega\,$ with $\,0<q<1\,$, $\,\omega\geq 0$, and
the Jackson $\,q$-integral, correlated with $\,\omega=0\,$ in the preceding case.
The corresponding fixed points are given by $\,s_0=\frac{\omega}{1-q}$ and $\,s_0=0\,$,
respectively.

Owing to the fact that
$\,\int_{a}^{b}f\,{\rm d}_{\beta}=-\int_{b}^{a}f\,{\rm d}_{\beta}\,$
(see \cite[Lemma 3.5, p.13]{HSSA:2015}), hereinafter we'll assume that $\,a<b\,$
when the $\,\beta$-integral $\,\int_{a}^{b}f\,{\rm d}_{\beta}\,$ is considered.

\subsection{Properties of the $\beta$-derivative and of the $\beta$-integral}

We return to the definition of the $\,\beta$-derivative operator
(\ref{beta-derivative}).

\noindent We remark that
$$
\lim_{\beta(t)\rightarrow t}D_{\beta}[f](t)=f^\prime(t)\,,
$$
whenever $\,f\,$ is differentiable at the point $\,t\in I$, thus, $\,D_{\beta}\,$
is a $\,\beta$-analogue of the classical derivative operator.

The $\,\beta$-derivative satisfies properties that may be viewed as
$\,\beta$-analogues which agree with the properties for the classical derivative.
To illustrate it, we remark that the quantum operator (\ref{beta-derivative}) is
linear, i.e., whenever $\,\alpha\,$ and $\,\beta\,$ are real or complex numbers,
\begin{equation*}%\label{prop-linear-beta-derivative}
D_{\beta}[\alpha f+\beta g](t)=\alpha D_{\beta}[f](t)+\beta D_{\beta}[g](t)\,,
\end{equation*}
and satisfies the following $\,\beta$-product rule: for $\,t\in I,$
\begin{equation}\label{beta-product-rule}
\begin{array}{lll}
D_{\beta}[f\cdot g](t) & = & D_{\beta}[f](t)\cdot g(t)+f\big(\beta(t)\big)\cdot D_{\beta}[g](t) \\ [0.8em]
 & = & D_{\beta}[g](t)\cdot f(t)+g\big(\beta(t)\big)\cdot D_{\beta}[f](t)
 \end{array}
\end{equation}
when $\,f\,$ and $\,g\,$ are $\,\beta$-differentiable in $\,I$.
Besides, $\,f\,$ will be the constant function such that $\,f(t)=f(s_0)\,$ for all
$\,t\in I\,$ every time $\,D_{\beta}[f](t)=0\,$ for all $\,t\in I\,$.
For the above-mentioned and other properties of the general quantum difference
operator $\,D_{\beta}\,$ see \cite{HSSA:2015}.
These equalities hold for all $\,t\neq s_0$, and also for $\,t=s_0\,$ each and
whensoever $\,f^\prime(s_0)\,$ and $\,g^\prime(s_0)$ exist.

\subsubsection{The fundamental theorem of $\,\beta$-calculus}

Next assertion, a $\,\beta$-analogue of the fundamental theorem of calculus,
can be consulted in \cite{C:2021}.

\vspace{0.5em}
\textbf{Theorem B}.
Let $\,\beta:I\to I\,$ be a function satisfying the set-up delineated in subsection
\ref{Subsect-2.1}.
Fix $\,a,b\in I$ and let $\,f:I\rightarrow\mathbb{C}\,$ be a function such that
$\,D_{\beta}[f]\in\mathscr{L}_{\beta}^1[a,b]$. Then:
\begin{enumerate}
\item[{\rm (i)}] The equality
\begin{equation*}%\label{beta-fundamental-theorem}
\int_{a}^bD_{\beta}[f]\,{\rm d}_{\beta}=
\left[f(s)-\lim_{k\rightarrow+\infty}f\big(\beta^{k}(s)\big)\right]_{s=a}^b
\end{equation*}
holds, presuming the involved limits exist.
\item[{\rm (ii)}] Furthermore, assuming that $\,a<s_0<b,$ if $\,f\,$ has a discontinuity
of first kind at $\,s_0\,$ then
$$
\int_{a}^bD_{\beta}[f]\,{\rm d}_{\beta}=f(b)-f(a)-\Big(f(s_0^+)-f(s_0^-)\Big).
$$
Needless to say, if $\,f\,$ is continuous at $\,s_0\,$ then
$$
\int_{a}^bD_{\beta}[f]\,{\rm d}_{\beta}=f(b)-f(a)\,.
$$
\end{enumerate}

\subsubsection{The $\,\beta$-integration by parts formula}

Currently, we enunciate the $\,\beta$-analogue of the integration by parts formula
\cite{C:2021}.

\vspace{0.5em}
\textbf{Theorem C}.
Let $\,\beta:I\to I\,$ be a function satisfying the conditions described in subsection
\ref{Subsect-2.1}.
Fix $\,a,b\in I$ and two functions $\,f:I\rightarrow\mathbb{C}\,$ and
$\,g:I\rightarrow\mathbb{C}.$ Then:
$$
\int_{a}^b f\cdot D_{\beta}[g]\,{\rm d}_{\beta}=
\left[(f\cdot g)(s)-\lim_{k\rightarrow+\infty}(f\cdot g)\big(\beta^k(s)\big)\right]_{s=a}^b
-\int_{a}^b\big(g\circ\beta\big)\cdot D_{\beta}[f]\,{\rm d}_{\beta}
$$
holds, provided $\,f,g\in\mathscr{L}_{\beta}^1[a,b]$, $\,D_{\beta}[f]\,$ and $\,D_{\beta}[g]\,$
are bounded in $\,[a,b]_{\beta}$, and the limits exist.

Additionally, if $\,f\,$ and $\,g\,$ are continuous at $\,s_0\,$ then
$$
\int_{a}^b f\cdot D_{\beta}[g]\,{\rm d}_{\beta}=
\Big[ f\cdot g\Big]_{a}^b-\int_{a}^b \big(g\circ\beta\big)\cdot D_{\beta}[f]\,{\rm d}_{\beta}\,.
$$

\subsubsection{Conditions for the \emph{monotony} property of the $\,\beta$-integral}

Next proposition gives a sufficient condition to guarantee the \emph{monotony}
property of the $\,\beta$-integral. It can be found in \cite[Prop. 2, page 559]{C:2021}.

\vspace{0.5em}
\textbf{Proposition D}.
If $\,s_0\,$ is the fixed point of $\,\beta\,$ and $\,a\,$ and $\,b\,$ are elements
of $\,I\,$ such that $\,a\leq s_0\leq b,$ then
\begin{equation*}%\label{monotony}
f\leq g\;\:\mbox{in}\;\:[a,b]_{\beta}\:\Longrightarrow\:
\int_{a}^{b}f\,{\rm d}_{\beta}\leq\int_{a}^{b}g\,{\rm d}_{\beta}\,.
\end{equation*}

\subsection{The spaces $\,\mathscr{L}_{\beta}^p[a,b]\,$ and $\,L_{\beta}^p[a,b]$}

\subsubsection{The space $\,\mathscr{L}_{\beta}^p[a,b]$}

For $\,a,b\in I$, representing by \,$\mathscr{L}_{\beta}^p[a,b]\,$ the collection of
functions $\,f:I\rightarrow\mathbb{\mathbb{C}}\,$ such that $\,|f|^p\,$ is
$\,\beta$-integrable in $\,[a,b]\,$, i.e.,
$$\mathscr{L}_{\beta}^p[a,b]=\left\{f:I\rightarrow\mathbb{\mathbb{C}}\;\Big|\;
\int_{a}^b|f|^p{\rm d}_{\beta}<\infty\right\}\;.$$
We consider too
$$\mathscr{L}_{\beta}^\infty[a,b]=\left\{f:I\rightarrow\mathbb{\mathbb{C}}\;\Big|\;
\sup_{k\in\mathbb{N}_0}\Big\{\big|f\big(\beta^k(a)\big)\big|,\big|f\big(\beta^k(b)\big)\big|\Big\}<\infty\right\}\;.$$
In \cite[Corollary 3.4]{C:2021}, it was shown to be true that, whenever
$\,a\leq s_0\leq b\,$ and $\,1\leq p\leq\infty$, then the set
$\,\mathscr{L}_{\beta}^p[a,b]$, with the common operations of addition of functions
and multiplication of a function by a number (real or complex), becomes a linear space
over $\,\mathbb{K}$.

\subsubsection{The space $L_{\beta}^p[a,b]$}

For $\,f,g\in\mathscr{L}_{\beta}^p[a,b]$, put down $\,f\sim g\,$ when
\begin{equation}\label{EquivRel}
f\big(\beta^k(a)\big)=g\big(\beta^k(a)\big)\quad \mbox{and}\quad f\big(\beta^k(b)\big)=g\big(\beta^k(b)\big)
\end{equation}
holds for all $\,k=0,1,2,\cdots$, i.e., we declare that $\,f\sim g\,$ if $\,f=g\,$ in
$\,[a,b]_{\beta}$.
Plainly, $\,\sim\,$ defines an equivalence relation in
$\,\mathscr{L}_{\beta}^p[a,b]$.
We will denote by $\,L_{\beta}^p[a,b]\,$ the corresponding quotient set:
$$
L_{\beta}^p[a,b]:=\mathscr{L}_{\beta}^p[a,b]\big/\sim\;.
$$
The next theorems can be accessed in \cite{C:2021}. In the impending, $p^\prime$
represents the conjugate exponent of a real number $p\geq1$, i.e.,
$\frac{1}{p}+\frac{1}{p^\prime}=1$, with the stipulation that $p^\prime=\infty$ if $p=1$.
A suchlike affirmation can also be found in \cite{HS:2015}.

\vspace{0.5em}
\noindent\textbf{Theorem E}.
Suppose that $a\leq s_0\leq b$.
\begin{itemize}
\item If $1<p<\infty$, then
\begin{equation*}%\label{Eq2.3}
\int_{a}^b\left|fg\right|\,{\rm d}_{\beta}\leq
\left(\int_{a}^b\left|f\right|^p{\rm d}_{\beta}\right)^{\frac{1}{p}}
\left(\int_{a}^b\left|g\right|^{p^\prime}{\rm d}_{\beta}\right)^{\frac{1}{p^\prime}}
\end{equation*}
holds every time
$f\in\mathscr{L}_{\beta}^p[a,b]$ and $g\in\mathscr{L}_{\beta}^{p^\prime}[a,b]$.
\item If $p=1$, then
$$\int_{a}^b\left|fg\right|{\rm d}_{\beta}\leq
\sup_{k\in\mathbb{N}_0}\Big\{\big|g\big(\beta^k(a)\big)\big|,\big|g\big(\beta^k(b)\big)\big|\Big\}
\int_a^b|f|\,{\rm d}_{\beta}$$
holds, on condition that $\,f\in\mathscr{L}_{\beta}^1[a,b]\,$ and
$\,g\in\mathscr{L}_{\beta}^\infty[a,b]$.
\end{itemize}

\vspace{0.5em}
\noindent\textbf{Theorem F}.
If $\,a\leq s_0\leq b\,$ and $\,1\leq p\leq \infty\,$ then
$\,L_{\beta}^{p}[a,b]\,$ is a normed linear space over $\,\mathbb{R}\,$ or
$\,\mathbb{C}\,$ with norm
\begin{equation}\label{beta-p-norm}
\|f\|_{L_{\beta}^p[a,b]}:=\left\{
\begin{array}{ccl}
\displaystyle\left(\int_a^b|f|^p\,{\rm d}_{\beta}\right)^{\frac 1 p}&{\rm if}& 1\leq p<\infty\; ; \\
\rule{0pt}{1.5em}
\displaystyle\sup_{k\in\mathbb{N}_0}\,\Big\{\left|f\big(\beta^k(a)\big)\right|,\left|f\big(\beta^k(b)\big)\right|\Big\} &{\rm if}& p=\infty\;.
\end{array}\right.
\end{equation}

\vspace{0.5em}
In the standard way, here $\,f\in\mathscr{L}_{\beta}^p[a,b]\,$ designates any
representative of the relative class of $\,L_{\beta}^p[a,b]\,$ to which it belongs.
We remark that, taking into account (\ref{beta-integral-a-b-explicitly}) and
(\ref{EquivRel}), the definition of the norm $\,\|f\|_{L_{\beta}^p[a,b]}\,$ is
independent of the selected representative.

\vspace{0.5em}
\noindent\textbf{Theorem G}.
If $\,a\leq s_0\leq b\,$ and $\,1\leq p\leq \infty$, then the ensuing holds:

{\rm (i)} $L_{\beta}^{p}[a,b],$ endowed with the norm (\ref{beta-p-norm}), is
a Banach space for $\,1\leq p\leq\infty$, which is separable if $\,1\leq p<\infty\,$
and reflexive if $\,1<p<\infty$.

{\rm (ii)} $L_{\beta}^2[a,b]$ is a Hilbert space in regard to the inner product
\begin{equation}\label{beta-inner-product}
\langle f,g\rangle_{\beta}:=\int_a^bf\,\overline{g}\,{\rm d}_{\beta}\; ,\quad f,g\in L_{\beta}^2[a,b]\; .
\end{equation}

\section{$\beta$-Gr\"uss-type inequality}\label{section-3}

In \cite[Theorem 1, p. 75]{D:1999}, S. S. Dragomir published a generalization of
the Gr\"uss-type inequality (\ref{Gruss-inequality-product}) within the context of
inner product spaces. The $\,\beta$-inner product (\ref{beta-inner-product}),
see (ii) of Theorem G,
\begin{equation*}
\langle f,g\rangle_{\beta}:=\int_a^bf\,\overline{g}\,{\rm d}_{\beta}\; ,\quad
f,g\in L_{\beta}^2[a,b]\,,
\end{equation*}
is the right one to be considered for the corresponding $\,\beta$-quantum setting.
So, for real or complex functions and within the $\,\beta$-quantum scenario, we can
state and prove the following $\,\beta$-Gr\"uss type inequality as a simple corollary
of Theorem 1 of \cite{D:1999}.
\begin{thm}\label{Gruss-inequality-theorem}
Let $\,f,\,g\,:\,[a,b]\to\mathbb{R}\,$ be absolutely $\,\beta$-integrable
functions on $\,[a,b]\,$, and suppose that there exist real numbers
$\,m\,$, $\,n\,$, $\,M\,$ and $\,N\,$ such that
\begin{equation}\label{inequalities-hypothesis}
m\leq f(x)\leq M\,,\quad n\leq g(x)\leq N,\quad \mbox{for all}\;\:x\in [a,b]_{\beta}\,.
\end{equation}
If  $\,a<s_0<b\,$, then
%\begin{equation}\label{generalized-Gruss-inequality}
%\begin{array}{l}
%\displaystyle\left|\frac{1}{b-a}\int_{a}^{b}f(x)\,g(x)\,d_{\beta}x-
%\frac{1}{b-a}\int_{a}^{b}f(x)\,d_{\beta}x\,\frac{1}{b-a}
%\int_{a}^{b}g(x)\,d_{\beta}x\right|\leq \\ [1em]
%\displaystyle\hspace{10em}\frac{1}{4}(M-m)(N-n)\,,
%\end{array}
%\end{equation}
\begin{equation}\label{Gruss-inequality}
\left|\frac{1}{b-a}\int_{a}^{b}\!f(x)\,g(x)\,d_{\beta}x-
\frac{1}{(b-a)^2}\int_{a}^{b}\!f(x)\,d_{\beta}x\int_{a}^{b}\!g(x)\,d_{\beta}x\right|
\leq\frac{1}{4}(M-m)(N-n)
\end{equation}
and the constant $\,\frac{1}{4}\,$ is the best possible.
\end{thm}
\begin{proof}
In order to follow the proof from \cite{D:1999}, notice that the hypothesis
inequalities (\ref{inequalities-hypothesis}) are equivalent to
$\,\big(M-f(x)\big)\big(f(x)-m\big)\geq 0\,$ and
$\,\big(N-g(x)\big)\big(g(x)-n\big)\geq 0\,$, which, by Propositon D, imply
\begin{equation*}
\int_{a}^{b}\big(M-f(x)\big)\big(f(x)-m\big)\,d_{\beta}x\geq 0\,,\quad
\int_{a}^{b}\big(N-g(x)\big)\big(g(x)-n\big)\,d_{\beta}x\geq 0\,.
\end{equation*}
\end{proof}
Further generalizations (see \cite[Proposition 1, p. 78]{D:1999}) could also be
considered in the set of all \emph{isotonic} (\cite[p. 77]{D:1999}) linear functionals
in $\,L_{\beta}^2[a,b]\,$, i.e., linear functionals which satisfy the
\emph{monotony} property (Proposition D) whenever
$\,a\leq s_0\leq b\,$. Accordingly,

\centerline{$\displaystyle f\longmapsto\frac{1}{b-a}\int_a^b f\,{\rm d}_{\beta}$}

\vspace{0.5em}
\noindent gives us a \emph{normalized} linear functional \cite[p. 77]{D:1999} in the space
$\,L_{\beta}^2[a,b]\,$, which is \emph{isotonic} whenever $\,a\leq s_0\leq b\,$,
where $\,s_0\,$ is the only fixed point of the function $\,\beta\,$.

Anyhow, we'll present a direct proof of Theorem
\ref{Gruss-inequality-theorem}. To achieve this goal, we'll explore the ideas
of \cite[Chapter X, p. 295]{MPF:1993}, which, in our opinion, will enrich this work
since, some of the needed results, despite being simple, are not yet published in
the literature. First of all, let us introduce the following definition.
\begin{defn}\label{Chebyshev-functional}
Being $\,f,\,g:\,[a,b]\to\mathbb{R}\,$ two absolutely $\,\beta$-integrable
functions on $\,[a,b]\,$, the Chebyshev functional
$\,T_{\beta}\,$ is defined by

$\displaystyle T_{\beta}(f,g):=\frac{1}{b-a}\int_{a}^{b}f(x)\,g(x)\,{\rm d}_{\beta}x\,-
\,\left(\frac{1}{b-a}\int_{a}^{b}f(x)\,{\rm d}_{\beta}x\right)
\left(\frac{1}{b-a}\int_{a}^{b}g(x)\,{\rm d}_{\beta}x\right)\,.$
\end{defn}
We remark that, the classical functional $\,T(f,g)\,$, defined by

$\displaystyle T(f,g):=\frac{1}{b-a}\int_{a}^{b}f(x)\,g(x)\,{\rm d}x\,-
\,\left(\frac{1}{b-a}\int_{a}^{b}f(x)\,{\rm d}x\right)
\left(\frac{1}{b-a}\int_{a}^{b}g(x)\,{\rm d}x\right)\,.$
in the context of the Riemann or the Lebesgue integration, is sometimes
identified as the \u{C}eby\u{s}ev functional.

\begin{prop}\label{positivity-proposition}
Assume that $\,f:\,[a,b]\to\mathbb{R}\,$ is an absolutely $\,\beta$-integrable
function on $\,[a,b]\,$. If $\,a\leq s_0\leq b\,$ then
\begin{equation}\label{positivity}
T_{\beta}(f,f)\geq 0\,.
\end{equation}
\end{prop}
\begin{proof}
In fact, for any real constant $\,c\,$, since $\,a\leq s_0\leq b\,$, by Proposition D,
$$0\leq\int_a^b\big(f(x)-c\big)^2\,{\rm d}_{\beta}x=\int_a^b f^2(x)\,{\rm d}_{\beta}x
\,-\,2c\int_a^b f(x)\,{\rm d}_{\beta}x+c^2(b-a)\,,$$
thus, for the particular case where $\,c=\frac{1}{b-a}\int_ a^b f(t)\,{\rm d}_{\beta}t\,$,
\begin{equation*}
0\leq\int_a^b\left(f(x)-\frac{\int_ a^b f(t)\,d_{\beta}t}{b-a}\right)^2\,{\rm d}_{\beta}x=
\int_a^b f^2(x)\,{\rm d}_{\beta}x-
\frac{1}{b-a}\left(\int_ a^b f(x)\,{\rm d}_{\beta}x\right)^2 \,,
\end{equation*}
therefore

\centerline{
$\displaystyle T_{\beta}(f,f)=\frac{1}{b-a}\int_a^b f^2(x)\,{\rm d}_{\beta}x-
\frac{1}{(b-a)^2}\left(\int_ a^b f(x)\,{\rm d}_{\beta}x\right)^2\geq 0\,.$
}
\end{proof}

Then, one may obtain the following identity.
\begin{lem}[$\beta$-Korkine identity]\label{beta-Korkine}
Let $\,f,\,g:\,[a,b]\to\mathbb{R}\,$ be two absolutely $\,\beta$-integrable functions
on $\,[a,b]\,$. Then
\begin{equation*}
T_{\beta}(f,g)=\frac{1}{2(b-a)^2}
\int_{a}^{b}\int_{a}^{b}\big(f(x)-f(y)\big)\big(g(x)-g(y)\big)\,
{\rm d}_{\beta}x\,{\rm d}_{\beta}y\,.
\end{equation*}
\end{lem}
\begin{proof}
Using the linearity property of the $\,\beta$-integral, we obtain
\begin{equation*}
\begin{array}{l}\displaystyle
\int_{a}^{b}\int_{a}^{b}\big(f(x)-f(y)\big)\big(g(x)-g(y)\big)\,
{\rm d}_{\beta}x\,{\rm d}_{\beta}y \\ [1.2em]
\hspace{2em} = \displaystyle
\int_{a}^{b}\int_{a}^{b}\Big(f(x)\,g(x)-f(x)\,g(y)-f(y)\,g(x)+f(y)\,g(y)\Big)\,
{\rm d}_{\beta}x\,{\rm d}_{\beta}y \\ [1em]
%\hspace{2em} = \displaystyle
%(b-a)\int_{a}^{b}f(x)\,g(x)\,{\rm d}_{\beta}x-\int_{a}^{b}f(x)\,{\rm d}_{\beta}x\,
%\int_{a}^{b}g(y)\,{\rm d}_{\beta}y \\ [1em]
%\hspace{5em}\displaystyle
%-\int_{a}^{b}f(y)\,{\rm d}_{\beta}y\,\int_{a}^{b}g(x)\,{\rm d}_{\beta}x+
%(b-a)\int_{a}^{b}f(y)\,g(y)\,{\rm d}_{\beta}y \\ [1.2em]
\hspace{2em} = \displaystyle
2(b-a)\int_{a}^{b}f(x)\,g(x)\,{\rm d}_{\beta}x-2\left(\int_{a}^{b}f(x)\,{\rm d}_{\beta}x\right)
\left(\int_{a}^{b}g(x)\,{\rm d}_{\beta}x\right)\,.
\end{array}
\end{equation*}
\end{proof}
The following proposition is a simple and direct consequence of Theorem E, so we
omit its proof.
\begin{prop}\label{double-Holder-inequality}
Assume that $\,p\,$ and $\,p^{\prime}\,$ are such that $1<p<\infty$, $1<p^{\prime}<\infty$
and $\,\frac{1}{p}+\frac{1}{p^{\prime}}=1\,$. If $\,a\leq s_0\leq b\,$ then
\begin{equation*}
\begin{array}{l}
\displaystyle
\int_{a}^b\int_{a}^b\left|f(x,y)\,g(x,y)\right|\,{\rm d}_{\beta}x\,{\rm d}_{\beta}y\leq \\
\hspace{3em}\displaystyle
\left(\int_{a}^b\int_{a}^b\left|f(x,y)\right|^p{\rm d}_{\beta}x\,
{\rm d}_{\beta}y\right)^{\frac{1}{p}}
\left(\int_{a}^b\int_{a}^b\left|g(x,y)\right|^{p^{\prime}}{\rm d}_{\beta}x\,
{\rm d}_{\beta}y\right)^{\frac{1}{p^{\prime}}}
\end{array}
\end{equation*}
holds whenever
$\,f(x,\cdot)\in\mathscr{L}_{\beta}^p[a,b]\,$,
$\,g(x,\cdot)\in\mathscr{L}_{\beta}^{p^{\prime}}[a,b]\,$,
$\,\int_{a}^{b}\left|f(x,y)\right|^p{\rm d}_{\beta}x\in\mathscr{L}_{\beta}[a,b]\,$,
and
$\,\int_{a}^b\left|g(x,y)\right|^{p^{\prime}}{\rm d}_{\beta}x\in\mathscr{L}_{\beta}[a,b]\,$.
\end{prop}

\begin{prop}\label{kind-of-Cauchy-Schwarz-proposition}
Assume that $\,f,\,g:\,[a,b]\to\mathbb{R}\,$ are two absolutely $\,\beta$-integrable
functions on $\,[a,b]\,$. If $\,a\leq s_0\leq b\,$ then
\begin{equation}\label{kind-of-Cauchy-Schwarz-inequality}
T_{\beta}^2(f,g)\leq T_{\beta}(f,f)\,T_{\beta}(g,g)\,.
\end{equation}
\end{prop}
\begin{proof}
By the $\,\beta$-Korkine Lemma \ref{beta-Korkine} and Proposition D, the Chebyshev
functional of Definition \ref{Chebyshev-functional} satisfies

\vspace{0.5em}
\centerline{
$\displaystyle\big|T_{\beta}(f,g)\big|\leq\frac{1}{2(b-a)^2}
\int_{a}^{b}\int_{a}^{b}\Big|\big(f(x)-f(y)\big)\big(g(x)-g(y)\big)\Big|\,
{\rm d}_{\beta}x\,{\rm d}_{\beta}y\,.$}

\vspace{0.5em}
\noindent By Proposition \ref{double-Holder-inequality}, with $\,p=2=p^{\prime}\,$,
\begin{equation}\label{kind-of-C-S-identity1}
\begin{array}{l}
\displaystyle
\displaystyle\big|T_{\beta}(f,g)\big|\leq \\
\hspace{0,5em}\displaystyle
\frac{1}{2(b-a)^2}\left(\int_{a}^{b}\!\int_{a}^{b}\!\big(f(x)-f(y)\big)^2\,{\rm d}_{\beta}x\,{\rm d}_{\beta}y\right)
^{\frac{1}{2}}\left(\int_{a}^{b}\!\int_{a}^{b}\!\big(g(x)-g(y)\big)^2\,
{\rm d}_{\beta}x\,{\rm d}_{\beta}y\right)^{\frac{1}{2}}.
\end{array}
\end{equation}
Notice that,
\begin{equation}\label{kind-of-C-S-identity2}
\begin{array}{l}
\displaystyle
\int_{a}^{b}\!\int_{a}^{b}\!\big(f(x)-f(y)\big)^2\,{\rm d}_{\beta}x\,{\rm d}_{\beta}y=
\int_{a}^{b}\!\int_{a}^{b}\!\Big(f^2(x)-2\,f(x)\,f(y)+f^2(y)\Big)\,
{\rm d}_{\beta}x\,{\rm d}_{\beta}y= \\ [1em]
\hspace{1.5em}\displaystyle=(b-a)\int_{a}^{b}f^2(x)\,{\rm d}_{\beta}x-
2\,\int_{a}^{b}f(x)\,{\rm d}_{\beta}x\,\int_{a}^{b}f(y)\,{\rm d}_{\beta}y+
(b-a)\int_{a}^{b}f^2(y){\rm d}_{\beta}y  \\ [1em]
%\hspace{1.5em}\displaystyle=2(b-a)\int_{a}^{b}f^2(x)\,{\rm d}_{\beta}x-
%2\,\int_{a}^{b}f(x)\,{\rm d}_{\beta}x\,\int_{a}^{b}f(x)\,{\rm d}_{\beta}x \\ [1em]
\hspace{1.5em}\displaystyle=2(b-a)^2\left(\frac{1}{b-a}\int_{a}^{b}f^2(x)\,{\rm d}_{\beta}x-
\frac{1}{b-a}\int_{a}^{b}f(x)\,{\rm d}_{\beta}x\,
\frac{1}{b-a}\int_{a}^{b}f(x)\,{\rm d}_{\beta}x \right) \\ [1.5em]
\hspace{1.5em}\displaystyle=2(b-a)^2\,T(f,f)\,,
\end{array}
\end{equation}
and, as well,
\begin{equation}\label{kind-of-C-S-identity3}
\int_{a}^{b}\!\int_{a}^{b}\!\big(g(x)-g(y)\big)^2\,{\rm d}_{\beta}x\,{\rm d}_{\beta}y=
2(b-a)^2\,T(g,g)\,.
\end{equation}
Joining together (\ref{kind-of-C-S-identity1}), (\ref{kind-of-C-S-identity2}) and
(\ref{kind-of-C-S-identity3}), we immediately obtain
(\ref{kind-of-Cauchy-Schwarz-inequality}).
\end{proof}
We are, now, in conditions to prove directly the $\,\beta$-Gr\"uss inequality
(\ref{Gruss-inequality}) of Theorem \ref{Gruss-inequality-theorem}.
\begin{proof}[Direct proof of Theorem \ref{Gruss-inequality-theorem}]
Using properties of the $\,\beta$-integral and simplifying, we get
%\begin{equation*}%\label{beta-Gruss-type-inequality-identity1}
%\begin{array}{l}
%\displaystyle\left(M-\frac{1}{b-a}\int_{a}^{b}f(x){\rm d}_{\beta}x\right)
%\left(\frac{1}{b-a}\int_{a}^{b}f(x){\rm d}_{\beta}x-m\right) \\ [1em]
%\hspace{15em}\displaystyle
%-\frac{1}{b-a}\int_{a}^{b}\big(M-f(x)\big)\big(f(x)-m\big){\rm d}_{\beta}x \\ [1.2em]
%\hspace{3em}\displaystyle=\frac{1}{b-a}\int_{a}^{b}f^2(x){\rm d}_{\beta}x-
%\frac{1}{(b-a)^2}\left(\int_{a}^{b}f(x){\rm d}_{\beta}x\right)^2=T_{\beta}(f,f) %\\ [1.8em]
%%\hspace{3em}\displaystyle=T_{\beta}(f,f)\,.
%\end{array}
%\end{equation*}
\begin{equation*}%\label{beta-Gruss-type-inequality-identity1}
\begin{array}{l}
\displaystyle\left(M-\frac{\int_{a}^{b}f(x){\rm d}_{\beta}x}{b-a}\right)
\left(\frac{\int_{a}^{b}f(x){\rm d}_{\beta}x}{b-a}-m\right)-
\frac{1}{b-a}\int_{a}^{b}\big(M-f(x)\big)\big(f(x)-m\big){\rm d}_{\beta}x \\ [1.2em]
\hspace{3em}\displaystyle=\frac{1}{b-a}\int_{a}^{b}f^2(x){\rm d}_{\beta}x-
\frac{1}{(b-a)^2}\left(\int_{a}^{b}f(x){\rm d}_{\beta}x\right)^2=T_{\beta}(f,f) %\\ [1.8em]
%\hspace{3em}\displaystyle=T_{\beta}(f,f)\,.
\end{array}
\end{equation*}
Under the hypothesis $\,m\leq f(x)\leq M\,$ and because $\,a\leq s_0\leq b\,$,
by Proposition D, we have

\vspace{0.3em}
\centerline{$\,\displaystyle
m\leq\frac{1}{b-a}\int_{a}^{b}f(x){\rm d}_{\beta}x\leq M\,$,}

\noindent therefore
\begin{equation}\label{extra-inequality-1}
T_{\beta}(f,f)\leq
\left(M-\frac{1}{b-a}\int_{a}^{b}f(x){\rm d}_{\beta}x\right)
\left(\frac{1}{b-a}\int_{a}^{b}f(x){\rm d}_{\beta}x-m\right)\,,
\end{equation}
and, in a similar way, because $\,n\leq g(x)\leq N\,$,
\begin{equation}\label{extra-inequality-2}
T_{\beta}(g,g)\leq
\left(N-\frac{1}{b-a}\int_{a}^{b}g(x){\rm d}_{\beta}x\right)
\left(\frac{1}{b-a}\int_{a}^{b}g(x){\rm d}_{\beta}x-n\right)\,.
\end{equation}
But, for any real $\,x\,$ and $\,y\,$, we know that
$\,\displaystyle 4\,x\,y\,\leq\,(x+y)^2\,$, so, from (\ref{extra-inequality-1}) and
(\ref{extra-inequality-2}), we obtain
$\,T_{\beta}(f,f)\leq\frac{(M-m)^2}{4}\,$ and
$\,T_{\beta}(g,g)\leq\frac{(N-n)^2}{4}\,$, respectively.
%\begin{equation}%\label{extra-inequality-3}
%\left(M-\frac{1}{b-a}\int_{a}^{b}f(x){\rm d}_{\beta}x\right)
%\left(\frac{1}{b-a}\int_{a}^{b}f(x){\rm d}_{\beta}x-m\right)\leq\frac{(M-m)^2}{4}\,,
%\end{equation}
%and
%\begin{equation}\label{extra-inequality-4}
%\left(N-\frac{1}{b-a}\int_{a}^{b}g(x){\rm d}_{\beta}x\right)
%\left(\frac{1}{b-a}\int_{a}^{b}g(x){\rm d}_{\beta}x-n\right)\leq\frac{(N-n)^2}{4}\,.
%\end{equation}
To finish the direct proof of the Theorem \ref{Gruss-inequality-theorem}, it
is enough to introduce this last inequalities into inequality
(\ref{kind-of-Cauchy-Schwarz-inequality}) of Proposition \ref{kind-of-Cauchy-Schwarz-proposition}.
\end{proof}

\begin{cor}\label{lemma-functional}
Let $\,f,\,g\,:\,[a,b]\to\mathbb{R}\,$ be absolutely $\,\beta$-integrable
functions on $\,[a,b]\,$, and suppose that there exist real
numbers $\,m\,$ and $\,M\,$ such that
\begin{equation*}%\label{lemma-functional-hypothesis}
m\leq f(x)\leq M\,,\quad \mbox{for all}\;\:x\in [a,b]_{\beta}\,.
\end{equation*}
If $\,a<s_0<b\,$, then
\begin{equation}\label{lemma-functional-inequality}
\left|T_{\beta}(f,g)\right|\leq\frac{1}{2}(M-m)\sqrt{T_{\beta}(g,g)}\,.
\end{equation}
\end{cor}
\begin{proof}
First, notice that, under the hypothesis of Theorem \ref{Gruss-inequality-theorem}
and the use of the Chebyshev functional of Definition \ref{Chebyshev-functional},
the corresponding inequality (\ref{Gruss-inequality}) becomes,
\begin{equation*}
\big|T_{\beta}(f,g)\big|\leq\frac{1}{4}(M-m(N-n)\,,
\end{equation*}
so, in particular,
\begin{equation}\label{lemma-required-inequality}
\left|T_{\beta}(f,f)\right|\leq\frac{1}{4}(M-m)^2\,.
\end{equation}
Thus, inequality (\ref{lemma-functional-inequality}) follows by combining
inequality (\ref{lemma-required-inequality}) with inequality
(\ref{kind-of-Cauchy-Schwarz-inequality}) of Proposition
\ref{kind-of-Cauchy-Schwarz-proposition}.
\end{proof}

Now, we present a $\,\beta$-analogue of the Gr\"uss-type inequalities
\cite[Theorem 3, p. 640]{M:2002}. It involves two consecutive inequalities.
\begin{thm}\label{Gruss-type-inequality-theorem}
Let $\,f,\,g\,:\,[a,b]\to\mathbb{R}\,$ be absolutely $\,\beta$-integrable
functions on $\,[a,b]\,$, and suppose that there exist real
numbers $\,m\,$ and $\,M\,$ such that
\begin{equation}\label{Gruss-type-hypothesis}
m\leq f(x)\leq M\,,\quad \mbox{for all}\;\:x\in [a,b]_{\beta}\,.
\end{equation}
If $\,a\leq s_0\leq b\,$, then
\begin{equation}\label{Gruss-type-inequality}
\begin{array}{lll}
\displaystyle\left|T_{\beta}(f,g)\right| & \leq & \displaystyle
\frac{1}{2}\big(M-m\big)\,\frac{1}{b-a}\int_{a}^{b}
\left|g(x)-\frac{1}{b-a}\int_a^bg(t)\,{\rm d}_{\beta}t\right|\,{\rm d}_{\beta}x \\ [1.4em]
 & \leq & \displaystyle\frac{1}{2}\big(M-m\big)\,\sqrt{T_{\beta}(g,g)}\,.
\end{array}
\end{equation}
\end{thm}
\begin{proof}
We'll explore the ideas followed to prove \cite[Theorem 3, p. 640]{M:2002}.
We separate the proof in two parts, each one corresponding to one of the
inequalities of (\ref{Gruss-type-inequality}). In the classical case, the first
one is sometimes called \emph{pre-Gr\"uss inequality}.
\begin{enumerate}
\item[(i)] Consider the function $\,\Theta\,$ defined by
\begin{equation*}%\label{auxiliary-function}
\Theta(x):=g(x)-\frac{1}{b-a}\int_a^bg(t)\,{\rm d}_{\beta}t\,,\quad x\in [a,b]_ {\beta}\,,
\end{equation*}
and the corresponding \emph{positive} and \emph{negative} parts given by
\begin{equation*}%\label{positive-negative-parts}
\Theta^+(x):=\frac{1}{2}\Big(\big|\Theta(x)\big|+\Theta(x)\Big)\,,\quad
\Theta^-(x):=\frac{1}{2}\Big(\big|\Theta(x)\big|-\Theta(x)\Big)\,,
\quad x\in\,[a,b]_ {\beta}\,.
\end{equation*}
One can easily see that
\begin{equation}\label{nonegativity-positive-negative-parts}
\Theta^+(x)\geq 0\,,\quad \Theta^-(x)\geq 0\,, \quad \forall x\in [a,b]_ {\beta}\,,
\end{equation}
and
\begin{equation}\label{properties-auxiliary-function}
\Theta^+(x)+\Theta^-(x)=\big|\Theta(x)\big|\,, \quad
\Theta^+(x)-\Theta^-(x)=\Theta(x)\,, \quad \forall x\in [a,b]_ {\beta}\,.
\end{equation}
A simple computation shows that
\begin{equation}\label{integral-auxiliary-function-null}
\int_a^b \Theta(x)\,{\rm d}_{\beta}x=
%\int_a^b\left(g(x)-\frac{1}{b-a}\int_a^bg(t)\,{\rm d}_{\beta}t\right)\,{\rm d}_{\beta}x
0\,,
\end{equation}
and also that
\begin{equation}\label{new-identity-for-Tchebyshev-functional}
T_{\beta}(f,g)=\frac{1}{b-a}\int_a^b f(x)\,\Theta(x)\,{\rm d}_{\beta}x\,.
\end{equation}
From (\ref{integral-auxiliary-function-null}) and (\ref{properties-auxiliary-function})
we obtain
\begin{equation*}%\label{integral-auxiliary-identities-1}
\int_a^b \Theta^+(x)\,{\rm d}_{\beta}x=
\int_a^b \Theta^-(x)\,{\rm d}_{\beta}x=
\frac{1}{2}\int_a^b \big|\Theta(x)\big|\,{\rm d}_{\beta}x\,.
\end{equation*}
Now, using the hypothesis (\ref{Gruss-type-hypothesis}) and inequalities
(\ref{nonegativity-positive-negative-parts}), we may obtain
$$m\,\Theta^+(x)\leq f(x)\,\Theta^+(x)\leq M\,\Theta^+(x)\,,\quad\forall
x\in [a,b]_{\beta}\,,$$
as well as
$$-M\,\Theta^-(x)\leq -f(x)\,\Theta^-(x)\leq
-m\,\Theta^-(x)\,,\quad \forall x\in [a,b]_ {\beta}\,.$$
Performing $\,\beta$-integration in $\,[a,b]\,$, respectively,
since $\,a\leq s_0\leq b\,$, then
\begin{equation}\label{integral-auxiliary-identities-2}
m\,\int_a^b \Theta^+(x)\,{\rm d}_{\beta}x\leq
\int_a^b f(x)\,\Theta^+(x)\,{\rm d}_{\beta}x\leq
M\,\int_a^b \Theta^+(x)\,{\rm d}_{\beta}x
\end{equation}
and
\begin{equation}%\label{integral-auxiliary-identities-3}
-M\,\int_a^b \Theta^-(x)\,{\rm d}_{\beta}x\leq
-\int_a^b f(x)\,\Theta^-(x)\,{\rm d}_{\beta}x\leq
-m\,\int_a^b \Theta^-(x)\,{\rm d}_{\beta}x\,.
\end{equation}
In a similar way, we may obtain
\begin{equation}%\label{integral-auxiliary-identities-4}
m\,\int_a^b \Theta^-(x)\,{\rm d}_{\beta}x\leq
\int_a^b f(x)\,\Theta^-(x)\,{\rm d}_{\beta}x\leq
M\,\int_a^b \Theta^-(x)\,{\rm d}_{\beta}x
\end{equation}
and
\begin{equation}\label{integral-auxiliary-identities-5}
-M\,\int_a^b \Theta^+(x)\,{\rm d}_{\beta}x\leq
-\int_a^b f(x)\,\Theta^+(x)\,{\rm d}_{\beta}x\leq
-m\,\int_a^b \Theta^+(x)\,{\rm d}_{\beta}x\,.
\end{equation}
Adding up (\ref{integral-auxiliary-identities-2})-(\ref{integral-auxiliary-identities-5}),
respectively, together with (\ref{properties-auxiliary-function}), we obtain
$$
-\frac{1}{2}\big(M-m\big)\int_a^b\big|\Theta(x)\big|\,{\rm d}_{\beta}x\leq
\int_a^b f(x)\,\Theta(x)\,{\rm d}_{\beta}x\leq
\frac{1}{2}\big(M-m\big)\int_a^b\big|\Theta(x)\big|\,{\rm d}_{\beta}x
$$
which, by (\ref{new-identity-for-Tchebyshev-functional}),
proves the first inequality of (\ref{Gruss-type-inequality}).

\item[(ii)] Using the Cauchy-Bunyakovsky-Schwarz inequality
(Theorem E with $\,p=2=p^{\prime}\,$), we may write
%\begin{equation}\label{integral-auxiliary-identities-4}
%\begin{array}{l}
%\displaystyle\int_a^b\left|g(x)-\frac{1}{b-a}\int_a^b g(t)\,{\rm d}_{\beta}t\right|\,{\rm d}_{\beta}x\leq \\ [1.8em]
%\hspace{3em}\displaystyle\leq\sqrt{(b-a)}\sqrt{
%\int_a^b\left(g(x)-\frac{1}{b-a}\int_a^b g(t)\,{\rm d}_{\beta}t\right)^2\,{\rm d}_{\beta}x}\,.
%\end{array}
%\end{equation}
\begin{equation}\label{integral-auxiliary-identities-4}
\int_a^b\left|g(x)-\frac{\int_a^b g(t)\,{\rm d}_{\beta}t}{b-a}\right|\,{\rm d}_{\beta}x\leq
\sqrt{(b-a)}\sqrt{
\int_a^b\left(g(x)-\frac{\int_a^b g(t)\,{\rm d}_{\beta}t}{b-a}\right)^2\,{\rm d}_{\beta}x}\,.
\end{equation}
During the proof of the inequality (\ref{positivity}) of Proposition
\ref{positivity-proposition}, we saw that
\begin{equation*}
\begin{array}{lll}
\displaystyle
\int_a^b\left(g(x)-\frac{\int_a^b g(t)\,{\rm d}_{\beta}t}{b-a}\right)^2{\rm d}_{\beta}x
& = & \displaystyle\int_a^b g^2(x)\,{\rm d}_{\beta}x-
\frac{1}{b-a}\left(\int_ a^b g(x)\,{\rm d}_{\beta}x\right)^2 \\ [1em]
 & = & \displaystyle (b-a)\,T_{\beta}(g,g)\,,
\end{array}
\end{equation*}
so, the use of this last identity in (\ref{integral-auxiliary-identities-4}) gives
\begin{equation}\label{integral-auxiliary-identities-6}
\int_a^b\left|g(x)-\frac{1}{b-a}\int_a^b g(t)\,{\rm d}_{\beta}t\right|\,{\rm d}_{\beta}x
\leq(b-a)\sqrt{\,T_{\beta}(g,g)}\,.
\end{equation}
Finally, introducing inequality (\ref{integral-auxiliary-identities-6}) into the
right-hand side of the first inequality of (\ref{Gruss-type-inequality}), enables
to prove its second inequality.
\end{enumerate}
\end{proof}
\begin{proof}[Simpler proof, only for the first inequality of the Theorem
\ref{Gruss-type-inequality-theorem}] So that the reader can
explore different techniques, we present a much simpler proof of
the first inequality that figures in (\ref{Gruss-type-inequality})
of the Theorem \ref{Gruss-type-inequality-theorem}. We follow the
idea behind the proof of \cite[Theorem 1]{L:2007}.
Recalling (\ref{integral-auxiliary-function-null}),

\vspace{0.5em}
\centerline{$\displaystyle
\int_a^b \Theta(x)\,{\rm d}_{\beta}x=\int_a^b
\left(g(x)-\frac{1}{b-a}\int_a^b g(t)\,{\rm d}_{\beta}t\right){\rm d}_{\beta}x=0\,,$}

\vspace{0.5em}
\noindent then, by (\ref{new-identity-for-Tchebyshev-functional}), we have
\begin{equation*}%\label{Gruss-type-inequality-2}
\begin{array}{lll}
\displaystyle\big|T_{\beta}(f,g)\big| & = & \displaystyle
\frac{1}{b-a}\,\left|\int_{a}^{b}f(x)\,g(x)\,{\rm d}_{\beta}x-
\frac{1}{b-a}\int_a^bf(x)\,{\rm d}_{\beta}x\,\int_a^bg(x)\,{\rm d}_{\beta}x\right| \\ [1.8em]
 & = & \displaystyle\frac{1}{b-a}\,\left|\int_{a}^{b}\left(g(x)-
\frac{1}{b-a}\int_a^b g(t)\,{\rm d}_{\beta}t\right)\,\left(f(x)-\frac{M+m}{2}\right)\right|
\,{\rm d}_{\beta}x \\ [1.8em]
 & \leq & \displaystyle\sup_{x\in [a,b]_{\beta}}\left|f(x)-\frac{M+m}{2}\right|\,
\frac{1}{b-a}\int_{a}^{b}\left|
g(x)-\frac{1}{b-a}\int_a^b g(t)\,{\rm d}_{\beta}t\right|{\rm d}_{\beta}x\,.
\end{array}
\end{equation*}
But, since the assumption $\,m\leq f(x)\leq M\,$ is equivalent to
$\,\left|f(x)-\frac{M+m}{2}\right|\leq\frac{M-m}{2}\,$, the first
inequality of (\ref{Gruss-type-inequality}) is proved.
\end{proof}

\section{$\beta$-Gr\"uss type inequalities in the context of the Riemann-Stieltjes
$\beta$-integral}\label{section-4}

\subsection{The Riemann-Stieltjes $\beta$-integral}%\label{subsect-1}

Originally the definition of the classical Riemann-Stieltjes integral was
associated to an increasing function. Later that notion was generalized
by using a function of bounded variation. Since every monotone function is of
bounded variation, we'll consider this more general latter approach.

Let $\,u\,$ be a function of bounded variation on the $\,\beta$-interval
$\,[a,b]_{\beta}\,$. Of course that, the classical bounded variation concept
of a function on $\,[a,b]\,$, implies the function to be of bounded variation on
$\,[a,b]_{\beta}\,$. Unless otherwise stated, we will
henceforth assume that $\,u\,$ satisfies this condition.
\begin{defn}\label{Riemann-Stieltjes-beta-integral}
We define the Riemann-Stieltjes $\beta$-integral of the function $\,f\,$ by
\begin{equation*}
\begin{array}{l}
\displaystyle\int_{a}^{b}f(x)\,d_{\beta}u(x):=\sum_{n=0}^{+\infty}f\big(\beta^k(b)\big)
\Big(u\big(\beta^k(b)\big)-u\big(\beta^{k+1}(b)\big)\Big)- \\ [0.5em]
\displaystyle\hspace{11em}\sum_{n=0}^{+\infty}f\big(\beta^k(a)\big)
\Big(u\big(\beta^k(a)\big)-u\big(\beta^{k+1}(a)\big)\Big).
\end{array}
\end{equation*}
We say that a function $\,f\,$ is Riemann-Stieltjes $\,\beta$-integrable on
$\,[a,b]\,$ (with respect to the function $\,u\,$) if
$\,\int_{a}^{b}f(x)d_{\beta}u(x)\,$ is a finite number.
\end{defn}
If $\,D_{\beta}[u](t)\equiv\big(D_{\beta}u\big)(t)=\frac{u(t)-u(\beta(t))}{t-\beta(t)}\,$
is bounded on $\,[a,b]_{\beta}\,$, in particular, when
$\,u^{\hspace{0.01em}\prime}\big(s_0\big)\,$ is finite, then we may write
$$u(t)-u\big(\beta(t)\big)=\big(D_{\beta}u\big)(t)\,\big(t-\beta(t)\big).$$
Therefore, comparing with the definition of the $\,\beta$-integral
(\ref{beta-integral-a-b})-(\ref{beta-integral-a-b-explicitly}), the above
definition for the Riemann-Stieltjes $\beta$-integral becomes
\begin{equation}\label{R-S-1}
\displaystyle\int_{a}^{b}f(x)d_{\beta}u(x)=
\int_{a}^{b}f(x)\big(D_{\beta}u\big)(x)\,d_{\beta}x\,.
\end{equation}

\subsection{$\beta$-Gr\"uss type inequality associated with the Riemann-Stieltjes
$\beta$-integral}%\label{subsection-2}

We state the following definition.
\begin{defn}%\label{beta-L-Lipschitz}
We say that a real function $\,u\,$ is $\,\beta-L$-Lipschitzian on
$\,[a,b]\,$ if
\begin{equation*}
|u(x)-u\big(\beta(x)\big)|\leq L|x-\beta(x)|\:,\quad
\mbox{for all}\;x\in [a,b]_{\beta}\,.
\end{equation*}
\end{defn}

\begin{rem}
\begin{enumerate}
\item If $\,u\,$ is $\,\beta-L$-Lipschitzian on $\,[a,b]\,$, then,
Definition \ref{Riemann-Stieltjes-beta-integral} of the Riemann-Stieltjes
$\,\beta$-integral, will be consistent when the function $\,f\,$ is
Riemann-Stieltjes $\,\beta$-integrable in $\,[a,b]\,$, since the difference
factors which appear in its definition are always of the form
$\,u\big(\beta^{k}(x)\big)-u\big(\beta^{k+1}(x)\big)\,$,
where $\,k\,$ is any nonnegative integer and $\,x=a\,$ or $\,x=b\,$. \\
Also, identity (\ref{R-S-1}) is clearly valid.
\item Nevertheless, notice that the definition of a $\,\beta-L$-Lipschitzian function
on $\,[a,b]\,$, does not necessarily imply the function to be continuous at
the fixed point $\,s_0\,$.
\end{enumerate}
\end{rem}

\begin{prop}\label{Abs-integral}
Let $\,f\,$ be absolutely $\,\beta$-integrable on $\,[a,b]\,$.
If $\,u\,$ is $\,\beta-L$-Lipschitzian on $\,[a,b]\,$ and $\,a\leq s_0\leq b\,$
then
\begin{equation}\label{Abs-integral-ineq}
\left|\int_{a}^{b}f(x)\,d_{\beta}u(x)\right|\leq L\int_{a}^{b}|f(x)|\,d_{\beta}x\,.
\end{equation}
\end{prop}
\begin{proof}
In fact, from Definition \ref{Riemann-Stieltjes-beta-integral},
\begin{equation}\label{Inequality-0}
\begin{array}{l}
\displaystyle\left|\int_{a}^{b}f(x)\,d_{\beta}u(x)\right|\leq
\sum_{n=0}^{+\infty}\Big|f\big(\beta^k(b)\big)\Big|
\Big|u\big(\beta^k(b)\big)-u\big(\beta^{k+1}(b)\big)\Big| \\ [0.5em]
\displaystyle\hspace{12em}+\sum_{n=0}^{+\infty}\Big|f\big(\beta^k(a)\big)\Big|
\Big|u\big(\beta^k(a)\big)-u\big(\beta^{k+1}(a)\big)\Big|\,,
\end{array}
\end{equation}
hence, because $\,u\,$ is $\beta-L$-Lipschitzian on $\,[a,b]\,$,
\begin{equation*}
\begin{array}{l}
\displaystyle\left|\int_{a}^{b}f(x)\,d_{\beta}u(x)\right|\leq
L\left[\:\sum_{n=0}^{+\infty}\Big|f\big(\beta^k(b)\big)\Big|\,
\Big|\beta^k(b)-\beta^{k+1}(b)\Big|\right. \\ [0.5em]
\displaystyle\hspace{12em}\left.+\sum_{n=0}^{+\infty}\Big|f\big(\beta^k(a)\big)\Big|\,
\Big|\beta^k(a)-\beta^{k+1}(a)\Big|\,\right]\,.
\end{array}
\end{equation*}
Now, since $\,a\leq s_0\leq b\,$ we know (see (13) and (14) of \cite[p. 562]{C:2021})
that the sequence $\,\left\{\beta^k(a)\right\}_{k\in\mathbb{N}}\,$ is strictly
increasing while $\,\left\{\beta^k(b)\right\}_{k\in\mathbb{N}}\,$ is strictly
decreasing.
Also because $\,f\,$ is absolutely $\,\beta$-integrable on $\,[a,b]\,$, then
we may write
\begin{equation*}
\begin{array}{l}
\displaystyle\left|\int_{a}^{b}f(x)\,d_{\beta}u(x)\right|\leq
L\left[\:\sum_{n=0}^{+\infty}\Big|f\big(\beta^k(b)\big)\Big|
\Big(\beta^k(b)-\beta^{k+1}(b)\Big)\right. \\ [0.5em]
\displaystyle\hspace{3em}\left.-\sum_{n=0}^{+\infty}\Big|f\big(\beta^k(a)\big)\Big|
\Big(\beta^k(a)-\beta^{k+1}(a)\Big)\,\right]= L\int_{a}^{b}\big|f(x)\big|\,d_{\beta}x\,. %\\ [2em]
%\hspace{8em}\displaystyle = L\int_{a}^{b}\big|f(x)\big|\,d_{\beta}x\,.
\end{array}
\end{equation*}
\end{proof}
Then, we may state the following theorem. In its proof, we'll use a similar
approach to the one followed in \cite{DF:1998}, which, in turn, uses the techniques
of \cite[Chapter X]{MPF:1993}.
\begin{thm}\label{Gruss-type-Riemann-Stieltjes-theorem}
Let $\,f:\,[a,b]\to\mathbb{R}\,$ be an absolutely
$\,\beta$-integrable function on $\,[a,b]\,$, with
$\,a<s_0<b\,$, and suppose that there exist real numbers
$\,m\,$ and $\,M\,$ such that
\begin{equation*}
m\leq f(x)\leq M,\quad \mbox{for all}\;\:x\in [a,b]_{\beta}\,.
\end{equation*}
If $\,u\,$ is $\,\beta-L$-Lipschitzian on $\,[a,b]\,$ and
$\,\displaystyle\lim_{n\to\infty}u\left(\beta^{n}(x)\right)\,$ exist, whenever
$\,x=a\,$ and $\,x=b\,$, then the following inequality
\begin{equation}\label{general-Gruss-type-inequality}
\left|\int_{a}^{b}f(x)\,d_{\beta}u(x)-\frac{u(b)-u(a)-\left(u(s_0^+)-u(s_0^-)\right)}{b-a}
\int_{a}^{b}f(t)\,d_{\beta}t\right|\leq\frac{1}{2}L(M-m)(b-a)
\end{equation}
holds, being the constant $\,\frac{1}{2}\,$ the best possible.
\end{thm}
\begin{proof}
Since $\,a<s_0<b\,$ and
$\,\displaystyle\lim_{n\to\infty}u\left(\beta^{n}(x)\right)\,$ exist when
$\,x=a\,$ or $\,x=b\,$, and, also, by (i) of Lemma 2.1 in \cite[p. 3]{HS:2015},
(i) of Theorem B and (13)-(14) of \cite{C:2021}, we conclude that
\begin{equation}\label{general-Beta-Fundamental-Theorem}
\begin{array}{c}
\displaystyle\int_{a}^{b}d_{\beta}u(x)=\int_{a}^{b}\left(D_{\beta}u\right)(x)d_{\beta}x= %\\ [1em]
u(b)-u(a)-\left(\displaystyle\lim_{n\to\infty}u\left(\beta^{n}(b)\right)-
\displaystyle\lim_{n\to\infty}u\left(\beta^{n}(a)\right)\right) \\ [1em]
\hspace{5em}\displaystyle =u(b)-u(a)-\left(u(s_0^+)-u(s_0^-)\right)\,,
\end{array}
\end{equation}
thus,
\begin{equation*}
\begin{array}{l}\label{general-identity}
\displaystyle
\left|\int_{a}^{b}f(x)\,d_{\beta}u(x)-\frac{u(b)-u(a)-\left(u(s_0^+)-u(s_0^-)\right)}{b-a}
\int_{a}^{b}f(t)\,d_{\beta}t\right|= \\ [1em]
\hfill\displaystyle
\left|\int_{a}^{b}\left(f(x)-\frac{\int_{a}^{b}f(t)\,d_{\beta}t}{b-a}\right)\,
d_{\beta}u(x)\right|\,.
\end{array}
\end{equation*}
Now, under the hypothesis on $\,u\,$, inequality (\ref{Abs-integral-ineq}) of
Proposition \ref{Abs-integral} holds. Hence,
\begin{equation}
\begin{array}{l}\label{general-inequality}
\displaystyle\left|\int_{a}^{b}f(x)\,d_{\beta}u(x)-\frac{u(b)-u(a)-\left(u(s_0^+)-u(s_0^-)\right)}{b-a}
\int_{a}^{b}f(t)\,d_{\beta}t\right|\leq \\ [1em]
\hfill\displaystyle
L\int_{a}^{b}\left|f(x)-\frac{\int_{a}^{b}f(t)\,d_{\beta}t}{b-a}\right|d_{\beta}x\,.
\end{array}
\end{equation}
Let us take the following definition:
$$I:=\frac{1}{b-a}
\int_{a}^{b}\left(f(x)-\frac{1}{b-a}\int_{a}^{b}f(t)\,d_{\beta}t\right)^2d_{\beta}x\,.
$$
Proceeding in a similar way followed to prove Proposition \ref{positivity-proposition},
we obtain
\begin{equation*}
I=\frac{1}{b-a}\int_{a}^{b}f^2(x)\,d_{\beta}x-
\left(\frac{1}{b-a}\int_{a}^{b}f(t)\,d_{\beta}t\right)^2\,,
\end{equation*}
therefore
$$I=\left(M-\frac{\int_{a}^{b}\!f(t)\,d_{\beta}t}{b-a}\right)\!
\left(\frac{\int_{a}^{b}\!f(t)\,d_{\beta}t}{b-a}-m\right)-
\frac{1}{b-a}\int_{a}^{b}\!\big(M-f(t)\big)\big(f(t)-m\big)\,d_{\beta}t\,.$$
Since $\,m\leq f(x)\leq M\,$ on $\,[a,b]_{\beta}\,$ and $\,a\leq s_0\leq b\,$
then, by \cite[Lemma 3.14, p. 17]{HSSA:2015},
$$\;\int_{a}^{b}\big(M-f(t)\big)\big(f(t)-m\big)\,d_{\beta}t\geq 0\,,$$
thus
\begin{equation}%\label{Inq-1}
I\leq \left(M-\frac{1}{b-a}\int_{a}^{b}f(t)\,d_{\beta}t\right)
\left(\frac{1}{b-a}\int_{a}^{b}f(t)\,d_{\beta}t-m\right)\,.
\end{equation}
We know that
$$\,(M-\theta)(\theta-m)\leq\frac 1 4\left[(M-\theta)+(\theta-m)\right]^2=
\frac 1 4\big(M-m\big)^2$$
thereby, on one hand we have
\begin{equation}\label{Inq-1}
I\leq \frac 1 4\big(M-m\big)^2\,.
\end{equation}
On the other hand, by the $\beta$-Cauchy-Buniakowski-Schwarz inequality
(take $p=2$ in Theorem E) one has
\begin{equation}\label{Inq-2}
I\geq \left[\frac{1}{b-a}\int_{a}^{b}\left|f(x)-
\frac{1}{b-a}\int_{a}^{b}f(t)\,d_{\beta}t\right|d_{\beta}x\right]^2\,,
\end{equation}
so, by (\ref{Inq-1}) and (\ref{Inq-2}) one obtains
\begin{equation*}
\int_{a}^{b}\left|f(x)-\frac{1}{b-a}\int_{a}^{b}f(t)\,d_{\beta}t\right|d_{\beta}x
\leq \frac 1 2 (M-m)(b-a)\,.
\end{equation*}
Finally, the introduction of this last inequality into (\ref{general-inequality}),
leads to (\ref{general-Gruss-type-inequality}).

\vspace{0.5em}
\noindent \textbf{Step 2}. \\
Now we prove, under the generality of inequality (\ref{general-Gruss-type-inequality}),
that we can not choose a constant smaller than $\,\frac 1 2$.
Consider the functions $\,u\,$ and $\,f\,$ defined on $\,[a,b]_{\beta}\,$ by
\begin{equation}\label{particular-choice}
u(x):=\left|x-\frac{a+b}{2}\right|\;, \quad f(x):=sgn\left(x-\frac{a+b}{2}\right)\,.
\end{equation}
We have $\,-1\leq f(x) \leq 1\,$ in $\,[a,b]_{\beta}\,$ so $\,M-m=2\,$, thus
\begin{equation*}
\displaystyle\big|\,u(x)-u\big(\beta(x)\big)\,\big|\,=\,
\left|\;\left|x-\frac{a+b}{2}\right|-\left|\beta(x)-\frac{a+b}{2}\right|\;\right|
\,\leq\,\big|\,x-\beta(x)\,\big|\;,\quad \forall\,x\in [a,b]_{\beta}\,,
\end{equation*}
which shows that $\,u\,$ is a $\,\beta-L$-Lipshitzian function with $\,L=1$.
Hence, the validity of Definition \ref{Riemann-Stieltjes-beta-integral} and of identity
(\ref{general-Beta-Fundamental-Theorem}) are guaranteed. As a consequence of
these facts, on one hand we have
\begin{equation}\label{Identity-1}
\frac 1 2 L(M-m)(b-a)=b-a
\end{equation}
and $\,u(b)=\frac{b-a}{2}=u(a)\,.$ On the other hand, by the continuity of
this particular choice of the function $\,u\,$, $\,u(s_0^+)=u(s_0^-)\,$, therefore,
\begin{equation*}
\begin{array}{l}
\displaystyle
\left|\int_{a}^{b}f(x)\,d_{\beta}u(x)-\frac{u(b)-u(a)-\left(u(s_0^+)-u(s_0^-)\right)}{b-a}
\int_{a}^{b}f(t)\,d_{\beta}t\right|= \\ [1.4em]
\hspace{5em}\displaystyle
\int_{a}^{b}sgn\left(x-\frac{a+b}{2}\right)\,d_{\beta}u(x)=
-\int_{a}^{\frac{a+b}{2}}\,d_{\beta}u(x)+\int_{\frac{a+b}{2}}^{b}\,d_{\beta}u(x)\,,
\end{array}
\end{equation*}
thus, using (\ref{general-Beta-Fundamental-Theorem}) for this particular case,
\begin{equation}
\begin{array}{c}\label{Identity-2}
\displaystyle
\left|\int_{a}^{b}f(x)\,d_{\beta}u(x)-\frac{u(b)-u(a)-\left(u(s_0^+)-u(s_0^-)\right)}{b-a}
\int_{a}^{b}f(t)\,d_{\beta}t\right|\,=\, \\ [1em]
\displaystyle\frac{b-a}{2}+\frac{b-a}{2}\,=\,b-a\,.
\end{array}
\end{equation}
Identities (\ref{Identity-1}) and (\ref{Identity-2}) proves that the equality is
attained for this particular choice of $\,u\,$ and $\,f\,$.
\end{proof}
\begin{rem}
We can enlarge the restriction $\,a<s_0<b\,$ of Theorem
\ref{Gruss-type-Riemann-Stieltjes-theorem} to $\,a\leq s_0\leq b\,$.
This last general hypothesis, together with the fixed restriction $\,a<b\,$ at
the end of subsection \ref{beta-integral}, implies a small change in the corresponding
statement: if $\,a=s_0\,$, then we go to the statement of the theorem and replace
$\,u\big(s_0^-\big)\,$ by $\,u\left(s_0\right)\,$, while if $\,b=s_0\,$, then we
replace $\,u\big(s_0^+\big)\,$ by $\,u\left(s_0\right)\,$.
Similar generalizations with the corresponding adaptations can be performed in
some of the next corollaries.
\end{rem}

\subsection{Some particular cases}%\label{subsection-3}
The following corollary is an immediate consequence of Theorem
\ref{Gruss-type-Riemann-Stieltjes-theorem}.
\begin{cor}\label{corollary-1}
Let $\,f:\,[a,b]\to\mathbb{R}\,$ be an absolutely $\,\beta$-integrable function
on $\,[a,b]\,$, with $\,a<s_0<b\,$, and suppose that there
exist real numbers $\,m\,$ and $\,M\,$ such that
\begin{equation*}
m\leq f(x)\leq M,\quad \mbox{for all}\;\:x\in [a,b]_{\beta}\,.
\end{equation*}
If $\,u\,$ is $\,\beta-L$-Lipschitzian on $\,[a,b]_{\beta}\,$ and
$\,u(s_0^+)=u(s_0^-)\,$, then the following inequality
\begin{equation*}%\label{Gruss-type-inequality-1}
\left|\int_{a}^{b}f(x)\,d_{\beta}u(x)-\frac{u(b)-u(a)}{b-a}
\int_{a}^{b}f(t)\,d_{\beta}t\right|\leq\frac{1}{2}L(M-m)(b-a)
\end{equation*}
holds, being the constant $\,\frac{1}{2}\,$ the best possible. In
particular, this is true when $\,u\,$ is $\,\beta-L$-Lipschitzian
on $\,[a,b]_{\beta}\,$ and continuous on $\,s_0\,$.
\end{cor}

Let us now consider the following definition.
\begin{defn}\label{L-Lipschitz}
We say that a real function $\,u\,$ is $\,L$-Lipschitzian on $\,[a,b]_{\beta}\,$ if
\begin{equation*}
|u(x)-u(y)|\leq L|x-y|\:,\quad \mbox{for all}\;\:x, y\in [a,b]_{\beta}\,.
\end{equation*}
\end{defn}
We would like to draw your attention to the observations in the following remark.
\begin{rem}\label{remark-on-continuity-at-s0}
If $\,u\,$ is $\,L$-Lipschitzian on $\,[a,b]_{\beta}\,$, then:
\begin{enumerate}
\item of course, $\,u\,$ it is also $\,\beta-L$-Lipschitzian on $\,[a,b]\,$;
\item $\,u\,$ will be of bounded variation on $\,[a,b]_{\beta}\,$ and so, the
Definition \ref{Riemann-Stieltjes-beta-integral} is reliable whenever $\,f\,$ is an absolutely
$\,\beta$-integrable function. Identity (\ref{R-S-1}) is also valid;
\item if we restrict ourselves to the set $\,[a,b]_{\beta}\,$ and consider the
topology induced by the usual topology of $\,\mathbb{R}\,$, then we may say that
$\,u\,$ is continuous at the fixed point $\,s_0\in\,[a,b]_{\beta}\,$;
\item within the context of the previous observation, $\,u\,$ is continuous at the
fixed point $\,s_0\in\,[a,b]_{\beta}\,$, so, by $\,(ii)\,$ of Theorem B,
we conclude that
\begin{equation}\label{Beta-Fundamental-Theorem}
\int_{a}^{b}D_{\beta}[u](t)\,d_{\beta}t=u(b)-u(a)\,.
\end{equation}
\end{enumerate}
\end{rem}
As a consequence of Definition \ref{L-Lipschitz} and of the Remark
\ref{remark-on-continuity-at-s0}, one may prove the next corollary of Theorem
\ref{Gruss-type-Riemann-Stieltjes-theorem}.
\begin{cor}\label{corollary-2}
Let $\,f:\,[a,b]\to\mathbb{R}\,$ be an absolutely $\,\beta$-integrable function
on $\,[a,b]\,$, with $\,a\leq s_0\leq b\,$, and suppose that there
exist real numbers $\,m\,$ and $\,M\,$ such that
\begin{equation*}
m\leq f(x)\leq M,\quad \mbox{for all}\;\:x\in [a,b]_{\beta}\,.
\end{equation*}
If $\,u\,$ is $\,L$-Lipschitzian on $\,[a,b]_{\beta}\,$, then the
following inequality
\begin{equation*}%\label{Gruss-type-inequality-2}
\left|\int_{a}^{b}f(x)\,d_{\beta}u(x)-\frac{u(b)-u(a)}{b-a}
\int_{a}^{b}f(t)\,d_{\beta}t\right|\leq\frac{1}{2}L(M-m)(b-a)
\end{equation*}
holds, being the constant $\,\frac{1}{2}\,$ the best possible.
\end{cor}

\begin{rem}
\begin{enumerate}
\item This corollary is a $\,\beta$-analogue of Theorem 2.1 in \cite[p.287]{DF:1998}.
\item One can also prove that we can not choose a constant smaller than $\,\frac 1 2$
directly in the last corollary, by considering the same functions $\,u\,$ and
$\,f\,$ defined by (\ref{particular-choice}),
%\begin{equation*}
%u(x):=\left|x-\frac{a+b}{2}\right|\;, \quad f(x):=sgn\left(x-\frac{a+b}{2}\right)\,,
%\end{equation*}
%Again we have $\,-1\leq f(x) \leq 1\,$ in $\,[a,b]_{\beta}\,$ so $\,M-m=2\,$,
and realizing that
\begin{equation*}
\displaystyle\big|\,u(x)-u(y)\,\big|\,=\,
\Big|\;\left|x-\frac{a+b}{2}\right|-\left|y-\frac{a+b}{2}\right|\;\Big|
\,\leq\,\big|\,x-y\,\big|\;,\quad \mbox{for all}\;x\in [a,b]_{\beta}\,,
\end{equation*}
which shows that $\,u\,$ is a $\,L$-Lipshitzian function with $\,L=1$, thus
$\,u\,$ is $\,\beta-L$-Lipshitzian function with $\,L=1$.
So, we can follow the same steps to prove that the equality is attained
for this particular choice of $\,u\,$ and $\,f\,$.
\end{enumerate}
\end{rem}

We can also state and prove the following corollaries.
\begin{cor}\label{corollary-3}
Let $\,f:\,[a,b]\to\mathbb{R}\,$ be an absolutely $\,\beta$-integrable function
on $\,[a,b]\,$, with $\,a<s_0<b\,$, and suppose that there exist real numbers
$\,m\,$ and $\,M\,$ such that
\begin{equation*}
m\leq f(x)\leq M,\quad \mbox{for all}\;\:x\in [a,b]_{\beta}\,.
\end{equation*}
If $\,D_{\beta}u\,$ is bounded on $\,[a,b]_{\beta}\,$ and
$\,\displaystyle\lim_{n\to\infty}u\left(\beta^{n}(x)\right)\,$ exist, whenever
$\,x=a\,$ and $\,x=b\,$, then the following inequality
\begin{equation}\label{Gruss-type-inequality-3}
\begin{array}{l}
\displaystyle\left|\int_{a}^{b}f(x)\,d_{\beta}u(x)-\frac{u(b)-u(a)-\left(u(s_0^+)-u(s_0^-)\right)}{b-a}
\int_{a}^{b}f(t)\,d_{\beta}t\right|\leq \\ [1em]
\hspace{8em}\displaystyle\frac{1}{2}\|D_{\beta}u\|_{\infty}(M-m)(b-a)\,,
\end{array}
\end{equation}
holds, where $\,\displaystyle\|D_{\beta}u\|_{\infty}:=
\sup_{t\in [a,b]_{\beta}}\left|\left(D_{\beta}u\right)(t)\right|\,$.
%\begin{equation*}
%\|D_{\beta}u\|_{\infty}:= \sup_{t\in
%[a,b]_{\beta}}\left|\left(D_{\beta}u\right)(t)\right|<\infty\,.
%\end{equation*}
The constant $\,\frac{1}{2}\,$ is the best possible.
\end{cor}
\begin{proof}
To prove this corollary first notice that, since
$\,\left|\big(D_{\beta}u\big)(t)\right|\leq\|D_{\beta}u\|_{\infty}\,$ on
$\,[a,b]_{\beta}\,$, then
$$\left|u(t)-u\big(\beta(t)\big)\right|\leq\|D_{\beta}u\|_{\infty}|t-\beta(t)|\;,
\quad\mbox{for all}\quad t\in[a,b]_{\beta}\,,$$
which shows that $\,u\,$ is $\,\beta-$L-Lipschitzian in $\,[a,b]_{\beta}\,$, with
$\,L=\|D_{\beta}u\|_{\infty}\,$, whenever $\,\|D_{\beta}u\|_{\infty}\,$ is finite.
Using this last inequality in (\ref{Inequality-0}), one obtains
\begin{equation}\label{similar-to-prop}
\left|\int_{a}^{b}f(x)\,d_{\beta}u(x)\right|\leq
\|D_{\beta}u\|_{\infty}\int_{a}^{b}|f(x)|\,d_{\beta}x\,.
\end{equation}
This inequality is similar to (\ref{Abs-integral-ineq}) of Proposition
\ref{Abs-integral}. Again, we can use (\ref{general-Beta-Fundamental-Theorem}) and,
so, one may obtain (\ref{Gruss-type-inequality-3}) by performing the same steps
we followed to prove Theorem \ref{Gruss-type-Riemann-Stieltjes-theorem} but,
instead of using Proposition \ref{Abs-integral}, we now make use of inequality
(\ref{similar-to-prop}) to obtain (\ref{general-inequality}) with $\,L\,$ replaced
by $\,\|D_{\beta}u\|_{\infty}\,$.
%One could follow another approach to obtain the previous inequality by using
%identity (\ref{R-S-1}).
%
%\vspace{1em}
%{\color{red} For this purpose, notice, also, that if $\,D_{\beta}u\,$ is bounded then
%$\,u\,$ is $\,\beta-L$-Lipschitzian, with $\,L=\|D_{\beta}u\|_{\infty}$.) \\
%Notamos que no caso da derivada cl{\'a}ssica ser limitada podemos usar o teorema de
%Lagrange e provar que nesse caso a fun\c{c}{\~a}o {\'e} Lipshitziana.}
\end{proof}

\begin{cor}\label{corollary-4}
Let $\,f:\,[a,b]\to\mathbb{R}\,$ be an absolutely $\,\beta$-integrable function
on $\,[a,b]\,$, with $\,a<s_0<b\,$, and suppose that there exist real numbers
$\,m\,$ and $\,M\,$ such that
\begin{equation*}
m\leq f(x)\leq M,\quad \mbox{for all}\;\:x\in [a,b]_{\beta}\,.
\end{equation*}
If $\,D_{\beta}u\,$ is bounded on $\,[a,b]_{\beta}\,$ and $\,u(s_0^+)=u(s_0^-)\,$,
then the following inequality
\begin{equation*}%\label{Gruss-type-inequality-4}
\left|\int_{a}^{b}f(x)\,d_{\beta}u(x)-\frac{u(b)-u(a)}{b-a}
\int_{a}^{b}f(t)\,d_{\beta}t\right|\leq\frac{1}{2}\|D_{\beta}u\|_{\infty}(M-m)(b-a)
\end{equation*}
holds, where
%\begin{equation*}
%\|D_{\beta}u\|_{\infty}:=
%\sup_{t\in [a,b]_{\beta}}\left|\left(D_{\beta}u\right)(t)\right|<\infty\,.
%\end{equation*}
$\,\displaystyle\|D_{\beta}u\|_{\infty}:=
\sup_{t\in [a,b]_{\beta}}\left|\left(D_{\beta}u\right)(t)\right|\,$.
The constant $\,\frac{1}{2}\,$ is the best possible.
\end{cor}

\begin{rem}
\begin{enumerate}
\item Corollaries \ref{corollary-3} and \ref{corollary-4} are $\,\beta$-analogues of
Corollary 2.2 of \cite[p. 290]{DF:1998}.
\item In the classical case, a function that has finite derivative at a point must be
continuous at that point. However, the boundedness of  $\,D_{\beta}u\,$ in
$\,[a,b]_{\beta}\,$ does not imply the continuity of $\,u\,$ at the point
$\,s_0\in[a,b]_{\beta}\,$. This is the reason why, $\,u(s_0^+)-u(s_0^-)\,$
figures in (\ref{Gruss-type-inequality-3}) of Corollary \ref{corollary-3}.
\end{enumerate}
\end{rem}

\begin{cor}\label{corollary-5}
Let $\,f:\,[a,b]\to\mathbb{R}\,$ be absolutely $\,\beta$-integrable function
on $\,[a,b]\,$, with $\,a\leq s_0\leq b\,$, and such that there exist
real numbers $\,m\,$ and $\,M\,$ satisfying
\begin{equation*}
m\leq f(x)\leq M,\quad \mbox{for all}\;\:x\in [a,b]_{\beta}\,.
\end{equation*}
If $\,g:\,[a,b]\to\mathbb{R}\,$ is a nonnegative function on
$\,[a,b]\,$ that is continuous at $s_0$, then the following inequality
\begin{equation}\label{Gruss-type-inequality-5}
\left|\int_{a}^{b}f(x)\,g(x)\,d_{\beta}x-\frac{1}{b-a}\int_{a}^{b}g(t)\,d_{\beta}t
\int_{a}^{b}f(t)\,d_{\beta}t\right|\leq\frac{1}{2}\|g\|_{\infty}(M-m)(b-a)
\end{equation}
holds, with
%\begin{equation*}
%\|g\|_{\infty}:=
%\sup_{t\in [a,b]_{\beta}}\left|g(t)\right|<\infty\,.
%\end{equation*}
$\,\displaystyle\|g\|_{\infty}:=\sup_{t\in [a,b]_{\beta}}\left|g(t)\right|<\infty\,$.
The constant $\,\frac{1}{2}\,$ is the best possible.
\end{cor}
\begin{proof}
Using Theorem 3.6 of \cite[p. 14]{HSSA:2015}, we conclude that the function $\,G\,$,
defined on $\,[a,b]\,$ by $\,\displaystyle G(x)=\int_{a}^{x}g(t)d_{\beta}t\,$, is
non-decreasing, continuous at $\,s_0\,$ and is such that
$$\big(D_{\beta}G\big)(x)=g(x)\;,\quad \mbox{for all}\;x\in [a,b]\,.$$
Being non-decreasing, then $\,G\,$ is a function of bounded variation.
Thus, considering $\,u=G\,$, this assures the definition of the Riemann-Stieltjes
$\,\beta$-integral given in Definition \ref{Riemann-Stieltjes-beta-integral}.
We also know that $\,D_{\beta}u=D_{\beta}G=g\,$ will be bounded on $\,[a,b]_{\beta}\,$,
%%%Notar que, como $f$ {\'e} cont{\'\i}nua em $s_0$, existe $\delta>0$ tal que $f$ {\'e} limitada
%%%em $]s_0-\delta,s_0+\delta[$. Por sua vez, para qualquer $\delta>0$,
%%%$\,[a,b]_{\beta}\setminus]s_0-\delta,s_0+\delta[\,$ {\'e} um conjunto finito, logo $f$ {\'e}
%%%a{\'\i} limitada.
and since $\,g\,$ is continuous at $\,s_0\,$, then $\,u(s_0^+)=u(s_0^-)\,$.
Therefore, we can use (\ref{R-S-1}) and (\ref{Beta-Fundamental-Theorem}) in
the Corollary \ref{corollary-4} to get (\ref{Gruss-type-inequality-5}) % Corollary \ref{corollary-5}.
\end{proof}
\begin{rem}%\label{integral-of-product-remark}
\begin{enumerate}
\item We stress that (\ref{Gruss-type-inequality-5}) can be written in the
following form:
\begin{equation}\label{Gruss-type-inequality-5-v2}
\left|\frac{1}{b-a}\int_{a}^{b}f(x)\,g(x)\,d_{\beta}x-
\frac{1}{b-a}\int_{a}^{b}f(x)\,d_{\beta}x\,\frac{1}{b-a}
\int_{a}^{b}g(x)\,d_{\beta}x\right|\leq\frac{1}{2}\|g\|_{\infty}(M-m)
\end{equation}
This suggest that, under the conditions on the function $\,g\,$, Corollary
\ref{corollary-5} establishes an upper bound for the absolute value of the
difference, between the $\,\beta$-\emph{average} of the product of two functions
in $\,[a,b]\,$, and the product of the corresponding
$\,\beta$-\emph{averages} of each function in $\,[a,b]\,$.
\item Inequality (\ref{Gruss-inequality}) of Theorem \ref{Gruss-inequality-theorem}
can be compared with (\ref{Gruss-type-inequality-5-v2}).
\end{enumerate}
\end{rem}

\begin{cor}\label{corollary-6}
Let $\,f:\,[a,b]\to\mathbb{R}\,$ be absolutely $\,\beta$-integrable function on
$\,[a,b]\,$, with $\,a\leq s_0\leq b\,$, such that there exist real numbers
$\,m\,$ and $\,M\,$ satisfying
\begin{equation*}
m\leq f(x)\leq M,\quad \mbox{for all}\;\:x\in [a,b]_{\beta}\,.
\end{equation*}
If $\,D_{\beta}f\,$ is bounded on $\,[a,b]_{\beta}\,$ and $\,f(a)\neq f(b)\,$,
then the following inequality
\begin{equation}\label{Gruss-type-inequality-6}
\left|\frac{f(a)+f(b)}{2}-\frac{1}{b-a}\int_{a}^{b}\frac{f(t)+f\big(\beta(t)\big)}{2}\,d_{\beta}t
\right|\leq\frac{1}{2}\frac{\|D_{\beta}f\|_{\infty}}{\left|f(b)-f(a)\right|}(M-m)(b-a)
\end{equation}
holds, being
%\begin{equation*}
%\|D_{\beta}f\|_{\infty}:=
%\sup_{t\in [a,b]_{\beta}}\left|\big(D_{\beta}f\big)(t)\right|<\infty\,.
%\end{equation*}
$\,\displaystyle\|D_{\beta}f\|_{\infty}:=
\sup_{t\in [a,b]_{\beta}}\left|\big(D_{\beta}f\big)(t)\right|\,$.
The constant $\,\frac{1}{2}\,$ is the best possible.
\end{cor}
\begin{proof}
First, by (\ref{beta-product-rule}), one gets
\begin{equation*}
D_{\beta}\left[f^2\right](t)=f(t)\left(D_{\beta}f\right)(t)+
f\big(\beta(t)\big)\left(D_{\beta}f\right)(t)=\left[f(t)+f\big(\beta(t)\big)\right]
\left(D_{\beta}f\right)(t)\,.
\end{equation*}
Then, using the $\,\beta$-integration by parts formula (Theorem C),
\begin{equation*}
\int_{a}^{b}f(t)\left(D_{\beta}f\right)(t)\,d_{\beta}t=f^2(b)-f^2(a)-
\int_{a}^{b}f\big(\beta(t)\big)\left(D_{\beta}f\right)(t)\,d_{\beta}t\,,
\end{equation*}
hence
\begin{equation}\label{Identity-useful}
\int_{a}^{b}\left[f(t)+f\big(\beta(t)\big)\right]\left(D_{\beta}f\right)(t)\,d_{\beta}t=
f^2(b)-f^2(a)\,.
\end{equation}
Notice, also, that
%\begin{equation*}
\[\begin{array}{l}
\displaystyle
\left|\int_{a}^{b}\left[f(t)+f\big(\beta(t)\big)\right]\left(D_{\beta}f\right)(t)\,d_{\beta}t
-\frac{f(b)-f(a)}{b-a}
\int_{a}^{b}\left[f(t)+f\big(\beta(t)\big)\right]\,d_{\beta}t\right| \\ [2em]
\displaystyle\leq \left|\int_{a}^{b}f(t)\left(D_{\beta}f\right)(t)\,d_{\beta}t
-\frac{f(b)-f(a)}{b-a}\int_{a}^{b}f(t)\,d_{\beta}t\right|+ \\ [1.2em]
\hspace{5em}\displaystyle\left|\int_{a}^{b}f\big(\beta(t)\big)\left(D_{\beta}f\right)(t)\,d_{\beta}t
-\frac{f(b)-f(a)}{b-a}\int_{a}^{b}f\big(\beta(t)\big)\,d_{\beta}t\right|\,.
\end{array}\]
%\end{equation*}
Now, in the last inequality, we use identity (\ref{Identity-useful}) in the left-hand
side and Corollary \ref{corollary-4} for each absolute value in the right-hand side,
to obtain
%\[\begin{array}{l}
%\displaystyle
%\left|f^2(b)-f^2(a)
%-\frac{f(b)-f(a)}{b-a}
%\int_{a}^{b}\left[f(t)+f\big(\beta(t)\big)\right]\,d_{\beta}t\right| \\ [1.6em]
%\hspace{5em}\displaystyle\leq \frac{\|D_{\beta}f\|_{\infty}}{2}(M-m)(b-a)+
%\frac{\|D_{\beta}f\|_{\infty}}{2}(M-m)(b-a) \\ [1.4em]
%\hspace{5em}\displaystyle = \|D_{\beta}f\|_{\infty}(M-m)(b-a)\,,
%\end{array}\]
\begin{equation*}
\left|f^2(b)-f^2(a)
-\frac{f(b)-f(a)}{b-a}
\int_{a}^{b}\left[f(t)+f\big(\beta(t)\big)\right]\,d_{\beta}t\right|
= \|D_{\beta}f\|_{\infty}(M-m)(b-a)\,,
\end{equation*}
thus, since $\,f(a)\neq f(b)\,$, one easily obtains (\ref{Gruss-type-inequality-6}).
%Notice that, above, when we used Corollary 4.11 in the 2nd absolute value,
%if $\,m\leq f(x)\leq M\,$ for all $\,x\in\,[a,b]_{\beta}\,$, then one also have
%$\,m\leq f\big(\beta(x)\big)\leq M\,$ for all $\,x\in\,[a,b]_{\beta}\,$.
\end{proof}
\begin{rem}
Corollary \ref{corollary-6} is a $\,\beta$-analogue of Corollary 2.3 of
\cite[p. 290]{DF:1998}
\end{rem}

\subsection{The particular cases corresponding the Jackson $\,q$-integral and to
the Jackson-Thomae-N\"orlund $\,(q,\omega)$-integral}
%\label{subsection-4}
%Taking $\,\beta(t)=t+\omega\,$, with $\,\omega\,$ fixed in $\,\mathbb{R}^+\,$,
%which correspond to the $\,\omega$-derivative (or, forward difference operator)
%and to the corresponding inverse operator (the N\"orlund integral),
%Theorem \ref{Gruss-type-Riemann-Stieltjes-theorem} transforms into the following corollary.
%\begin{cor}\label{Gruss-type-ineq-Norlund}
%Let $\,f:\,[a,+\infty[\to\mathbb{R}\,$ be an absolutely
%$\,\triangle$-integrable function on $\,[a,b]\,$ and suppose
%that there exist real numbers $\,m\,$ and $\,M\,$ such that
%\begin{equation*}
%m\leq f(x)\leq M,\;\: \mbox{for all}\;\:x\in [a,b]_{\triangle}=
%\left\{a+n\omega:\,n\in\mathbb{N}_{0}\right\}\cup\left\{b+n\omega:\,n\in\mathbb{N}_{0}\right\}
%\cup\left\{s_0\right\}
%\end{equation*}
%%and
%%\begin{equation*}
%%\int_{a}^{b}f(x)\,\triangle_{\beta}x\geq 0\quad\mbox{whenever}\quad
%%f\geq 0\quad\mbox{in}\:[a,b]_{\triangle}\,.
%%\end{equation*}
%If $\,u\,$ is $\,\triangle-L$-Lipschitzian on $\,[a,b]_{\triangle}\,$ and
%$\,\displaystyle\lim_{x\to\infty}u(x)\,$ is finite, then the following inequality
%\begin{equation*}%\label{Gruss-type-inequality-Norlund}
%\left|\int_{a}^{b}f(x)\,\triangle_{\omega}u(x)-\frac{u(b)-u(a)}{b-a}
%\int_{a}^{b}f(t)\,\triangle_{\omega}t\right|\leq\frac{1}{2}L(M-m)(b-a)
%\end{equation*}
%holds.
%\end{cor}

If one considers the case where $\,\beta(t)=qt\,$, with $\,q\,$ fixed in $\,]0,1[\,$,
which correspond to the Jackson operator (\ref{Jackson-operator}) and to the
corresponding inverse operator (the Jackson $\,q$-integral),
Theorem \ref{Gruss-type-Riemann-Stieltjes-theorem} transforms into the following corollary.
\begin{cor}\label{Gruss-type-ineq-Jackson}
Let $\,f:\,[a,b]\to\mathbb{R}\,$ be an absolutely
$\,q$-integrable function on $\,[a,b]\,$, with
$\,a<0<b\,$, and suppose that there exist real numbers
$\,m\,$ and $\,M\,$ such that
\begin{equation*}
m\leq f(x)\leq M,\quad \mbox{for all}\;\:x\in [a,b]_{q}=
\left\{aq^n:\,n\in\mathbb{N}_{0}\right\}\cup\left\{bq^n:\,n\in\mathbb{N}_{0}\right\}
\cup\left\{s_0\right\}\,.
\end{equation*}
If $\,u\,$ is $\,q-L$-Lipschitzian on $\,[a,b]_q\,$ and
$\,\displaystyle\lim_{n\to\infty}u\left(q^{n}x\right)\,$ exist, whenever
$\,x=a\,$ and $\,x=b\,$, then the following inequality
\begin{equation*}%\label{Gruss-type-inequality-Jackson}
\left|\int_{a}^{b}f(x)\,d_{q}u(x)-\frac{u(b)-u(a)-\left(u(0^+)-u(0^-)\right)}{b-a}
\int_{a}^{b}f(t)\,d_{q}t\right|\leq\frac{1}{2}L(M-m)(b-a)
\end{equation*}
holds, being the constant $\,\frac{1}{2}\,$ the best possible.
\end{cor}

The case $\,\beta(t)=qt+\omega\,$, with $\,q\,$ fixed in $\,]0,1[\,$ and
$\,\omega\geq 0\,$, which correspond to the Hahn's quantum operator and to the
corresponding inverse operator (the Jackson-Thomae-N\"orlund $\,(q,\omega)$-integral),
Theorem \ref{Gruss-type-Riemann-Stieltjes-theorem} becomes the following corollary.
\begin{cor}\label{Gruss-type-ineq-Jackson-Thomae}
Let $\,f:\,[a,b]\to\mathbb{R}\,$ be an absolutely
$\,(q,\omega)$-integrable function on $\,[a,b]\,$, with
$\,a<\omega_0<b\,$, where $\,\omega_0=\frac{\omega}{1-q}\,$, and suppose
that there exist real numbers $\,m\,$ and $\,M\,$ such that
\begin{equation*}
m\leq f(x)\leq M,\quad \mbox{for all}\;\:x\in [a,b]_{q,\omega}\,,
\end{equation*}
with $\,\displaystyle[a,b]_{q,\omega}=
\left\{a\,q^n+\frac{1-q^n}{1-q}\omega:\,n\in\mathbb{N}_{0}\right\}\cup
\left\{b\,q^n+\frac{1-q^n}{1-q}\omega:\,n\in\mathbb{N}_{0}\right\}\cup\left\{s_0\right\}$\,. \\ [0.2em]
If $\,u\,$ is $\,(q,\omega)-L$-Lipschitzian on $\,[a,b]_{q,\omega}\,$ and
$\,\displaystyle\lim_{n\to\infty}u\left(x\,q^n+\frac{1-q^n}{1-q}\omega\right)\,$ exist,
whenever $\,x=a\,$ and $\,x=b\,$, then the following inequality
\begin{equation*}%\label{Gruss-type-inequality-Jackson-Thomae}
\left|\int_{a}^{b}f(x)\,d_{q,\omega}u(x)-\frac{u(b)-u(a)-\left(u(\omega_0^+)-u(\omega_0^-)\right)}{b-a}
\int_{a}^{b}f(t)\,d_{q,\omega}t\right|\leq\frac{1}{2}L(M-m)(b-a)
\end{equation*}
holds, being the constant $\,\frac{1}{2}\,$ the best possible.
\end{cor}

\section{Applications}\label{appls}
%\subsection{Inequalities involving the expected value}
To begin with, let's take the notation
\begin{equation*}%\label{Notation}
p_{a,b}:=\frac{1}{b-a}\int_{a}^{b}x\,d_{\beta}x\,,
\end{equation*}
considered in \cite[p. 5]{CS:2024}. Fix $\,a\,$ and $\,b\,$ in $\,\mathbb{R}\,$
such that $\,a<s_0<b\,$. Let $\,X\,$ be a
discrete \emph{random variable} with values in the discrete
\emph{sample space}
$\:S_{\beta}=\left\{\beta^{k}(a):\:k\in\mathbb{N}_0\right\}\cup
\left\{\beta^{k}(b):\:k\in\mathbb{N}_0\right\}\,$,
and suppose that the probability of each element to occur is given by
%\begin{equation}\label{probability-of-each-element-a}
%p_k(a)\equiv p_k\big(\beta^{k}(a)\big):=\frac{1}{b-a}
%\Big(\beta^{k+1}(a)-\beta^{k}(b)\Big)\,,\quad k=0,1,2,\cdots\,,
%\end{equation}
%and
%\begin{equation}\label{probability-of-each-element-b}
%p_k(b)\equiv p_k\big(\beta^{k}(b)\big):=\frac{1}{b-a}
%\Big(\beta^{k}(b)-\beta^{k+1}(b)\Big)\,,\quad k=0,1,2,\cdots\,.
%\end{equation}
\begin{equation}\label{probability-of-each-element}
\left\{\begin{array}{c}
p_k(a)\equiv p_k\big(\beta^{k}(a)\big):=\frac{1}{b-a}
\Big(\beta^{k+1}(a)-\beta^{k}(b)\Big)\,,\quad k=0,1,2,\cdots\,, \\ [1em]
p_k(b)\equiv p_k\big(\beta^{k}(b)\big):=\frac{1}{b-a}
\Big(\beta^{k}(b)-\beta^{k+1}(b)\Big)\,,\quad k=0,1,2,\cdots\,.
\end{array}\right.
\end{equation}
Notice that
\[
\begin{array}{l}
\displaystyle
\sum_{k=0}^{\infty}p_k(a)+\sum_{k=0}^{\infty}p_k(b) \; = \; \displaystyle
\sum_{k=0}^{\infty}p_k\big(\beta^{k}(a)\big)+
\sum_{k=0}^{\infty}p_k\big(\beta^{k}(b)\big) \\ [1.5em]
\hspace{5em}  = \; \displaystyle
\frac{1}{b-a}\sum_{k=0}^{\infty}\Big(\beta^{k+1}(a)-\beta^{k}(a)\Big)+
\frac{1}{b-a}\sum_{k=0}^{\infty}\Big(\beta^{k}(b)-\beta^{k+1}(b)\Big) \\ [1.5em]
\hspace{5em}  = \; \displaystyle\frac{1}{b-a}\Big(\lim_{k\to\infty}\beta^{k+1}(a)-a\Big)+
\frac{1}{b-a}\Big(b-\lim_{k\to\infty}\beta^{k+1}(b)\Big) \\ [1.5em]
\hspace{5em} = \; \displaystyle\frac{1}{b-a}\big(s_0-a\big)+\frac{1}{b-a}\big(b-s_0\big)=1 %\\ [1.2em]
%\hspace{5em} = \; 1\,,
\end{array}
\]
and, also,
\[
\begin{array}{l}
\displaystyle\sum_{k=0}^{\infty}\beta^{k}(b)\,p_k\big(\beta^{k}(b)\big)+
\sum_{k=0}^{\infty}\beta^{k}(a)\,p_k\big(\beta^{k}(a)\big) \\ [1.5em]
\hspace{2em} = \;\displaystyle
\frac{1}{b-a}\sum_{k=0}^{\infty}\beta^{k}(b)\Big(\beta^{k}(b)-\beta^{k+1}(b)\Big)-
\frac{1}{b-a}\sum_{k=0}^{\infty}\beta^{k}(a)\Big(\beta^{k}(a)-\beta^{k+1}(a)\Big) \\ [1.5em]
\hspace{2em} = \;\displaystyle
\frac{1}{b-a}\int_{s_0}^{b}x\,d_qx-\frac{1}{b-a}\int_{s_0}^{a}x\,d_qx\,=\,
\frac{1}{b-a}\int_{a}^{b}x\,d_qx\,=\,p_{a,b}\,. %\\ [1.5em]
%\hspace{2em} = \;\displaystyle\frac{1}{b-a}\int_{a}^{b}x\,d_qx \\ [1em]
%\hspace{2em} = \;\displaystyle p_{a,b}\,.
\end{array}
\]
Therefore, for the particular choice made in (\ref{probability-of-each-element}),
$\,\displaystyle p_{a,b}\,$ represents the \emph{expected value} of the
\emph{random variable} $\,X\,$ whose values are in the given \emph{sample space} $\,S_{\beta}\,$.

Under the above setting, suppose that $\,f\,$ and $\,g\,$ are two functions
defined on $\,S_{\beta}\,$, both $\,\beta$-integrable on $\,[a,b]\,$. Then,
$$\frac{1}{b-a}\int_{a}^{b}f(x)\,d_qx\;,\quad\frac{1}{b-a}\int_{a}^{b}g(x)\,d_qx
\quad\mbox{and}\quad\frac{1}{b-a}\int_{a}^{b}f(x)\,g(x)\,d_qx\,,$$
gives the \emph{expected values} of $\,f\big(X\big)\,$, $\,g\big(X\big)\,$ and
$\,(f\,g)\big(X\big)\,$, respectively, when $\,X\,$ runs throughout the
\emph{sample space} $\,S_{\beta}\,$.
Within this scenario, if we admit that $\,f\,$ and $\,g\,$, respectively, satisfy
(\ref{inequalities-hypothesis}) of Theorem \ref{Gruss-inequality-theorem}, then we obtain
the inequality
\begin{equation}\label{statistic-Gruss-inequality-1}
\left|\frac{1}{b-a}\int_{a}^{b}\!f(x)\,g(x)\,d_{\beta}x-
\frac{1}{(b-a)^2}\int_{a}^{b}\!f(x)\,d_{\beta}x\int_{a}^{b}\!g(x)\,d_{\beta}x\right|
\leq\frac{1}{4}(M-m)(N-n)\,,
\end{equation}
which compares the \emph{expected value} of $\,(f\,g)\big(X\big)\,$ with the product of the
\emph{expected values} of $\,f\big(X\big)\,$ and $\,g\big(X\big)\,$.

Rewriting (\ref{statistic-Gruss-inequality-1}) as
\[
\begin{array}{l}
\displaystyle
\frac{1}{(b-a)^2}\int_{a}^{b}\!f(x)\,d_{\beta}x\int_{a}^{b}\!g(x)\,d_{\beta}x-
\frac{1}{4}(M-m)(N-n)\leq\frac{1}{b-a}\int_{a}^{b}\!f(x)\,g(x)\,d_{\beta}x \\ [1.5em]
\hspace{3em}\displaystyle
\leq\frac{1}{(b-a)^2}\int_{a}^{b}\!f(x)\,d_{\beta}x\int_{a}^{b}\!g(x)\,d_{\beta}x+
\frac{1}{4}(M-m)(N-n)\,,
\end{array}
\]
and, additionally, if $\,f\,$ and $\,g\,$ are convex function on $\,[a,b]\,$,
then, by the $\,\beta$-Hermite-Hadamard Theorem 3.1 of \cite[p.7]{CS:2024},
we obtain
\begin{equation}\label{statiscal-example}
\begin{array}{l}
\displaystyle
f\left(p_{a,b}\right)\,g\left(p_{a,b}\right)-
\frac{1}{4}(M-m)(N-n)\leq\frac{1}{b-a}\int_{a}^{b}\!f(x)\,g(x)\,d_{\beta}x\leq \\ [1.5em]
\displaystyle
\!\left[\!\left(1-\frac{p_{a,b}-a}{b-a}\right)\!f(a)+\frac{p_{a,b}-a}{b-a}f(b)\right]
\!\left[\left(1-\frac{p_{a,b}-a}{b-a}\right)\!g(a)+\frac{p_{a,b}-a}{b-a}g(b)\right] \\ [1.5em]
+\frac{1}{4}(M-m)(N-n)\,,
\end{array}
\end{equation}
which gives a lower and an upper bound for the \emph{expected
value} of $\,(f\,g)\big(X\big)\,$, when $\,X\,$ runs throughout
the \emph{sample space} $\,S_{\beta}\,$.

In the particular case of the Jackson $\,q$-integral, where
$\,\beta(t)=qt\,$, with $\,q\,$ fixed in the interval $\,]0,1[\,$,
we have $\,p_{\hspace{0.05em}a,b}=\frac{a+b}{1+q}\,$. Thus,
(\ref{statiscal-example}) becomes
\begin{equation*}
\begin{array}{l}
\displaystyle
f\left(\frac{a+b}{1+q}\right)\,g\left(\frac{a+b}{1+q}\right)-
\frac{1}{4}(M-m)(N-n)\leq\frac{1}{b-a}\int_{a}^{b}\!f(x)\,g(x)\,d_{q}x\leq \\ [1.5em]
\displaystyle
\!\left[\frac{-a+qb}{(1+q)(b-a)}f(a)+\frac{-qa+b}{(1+q)(b-a)}f(b)\right]\!
\left[\frac{-a+qb}{(1+q)(b-a)}g(a)+\frac{-qa+b}{(1+q)(b-a)}g(b)\right] \\ [1.5em]
+\frac{1}{4}(M-m)(N-n)\,.
\end{array}
\end{equation*}

\end{document}